\documentclass[11pt]{amsart}
\usepackage[T1]{fontenc}
\usepackage{amssymb}
\usepackage{amsmath}
\usepackage{amsthm}
\usepackage{fancyhdr}
\usepackage{mathtools}
\usepackage[british]{babel}
\usepackage{geometry}
\usepackage{enumitem}
\usepackage{algpseudocode}
\usepackage{dsfont}
\usepackage{centernot}
\usepackage{xstring}
\usepackage{colortbl}
\usepackage[section]{algorithm}
\usepackage{graphicx}
\usepackage[font={footnotesize}]{caption}
\usepackage[usenames,dvipsnames,table]{xcolor}
\usepackage{pstricks,tikz}
\usepackage{xcolor}
\usepackage[h]{esvect}
\usepackage[noadjust]{cite}
\usepackage[
		bookmarksopen=true,
		bookmarksopenlevel=1,
		colorlinks=true,
		linkcolor=darkblue,
        linktoc=page,
		citecolor=darkblue,
]{hyperref}

\title[]{Counting Hamilton cycles in Dirac hypergraphs}
\date{\today}

\author[Stefan~Glock]{Stefan~Glock}

\author[Stephen~Gould]{Stephen~Gould}

\author[Felix~Joos]{Felix~Joos}

\author[Daniela~K\"uhn]{Daniela~K\"uhn}

\author[Deryk~Osthus]{Deryk~Osthus}

\thanks{This project has received partial funding from the European Research
Council (ERC) under the European Union's Horizon 2020 research and innovation programme (grant agreement no. 786198, D.~K\"uhn and D.~Osthus).
The research leading to these results was also partially supported by the EPSRC, grant nos. EP/N019504/1 (S.~Glock and D.~K\"uhn) and EP/S00100X/1 (D.~Osthus), and the Deutsche Forschungsgemeinschaft (DFG, German Research Foundation), grant nos. 339933727 and 428212407 (F.~Joos), as well as the Royal Society and the Wolfson Foundation (D.~K\"uhn).}

\newtheorem{theorem}{Theorem}[section]
\newtheorem{prop}[theorem]{Proposition}
\newtheorem{lemma}[theorem]{Lemma}
\newtheorem{cor}[theorem]{Corollary}

\newtheorem{conj}[theorem]{Conjecture}

\theoremstyle{definition}

\newtheoremstyle{claimstyle}{5pt}{5pt}{\em}{5pt}{\em}{:}{5pt}{}
\theoremstyle{claimstyle}

\newtheoremstyle{stepstyle}{10pt}{5pt}{\em}{0pt}{\em}{:}{5pt}{}
\theoremstyle{stepstyle}

\numberwithin{equation}{section}

\definecolor{darkblue}{rgb}{0,0,0.5}

\def\noproof{{\unskip\nobreak\hfill\penalty50\hskip2em\hbox{}\nobreak\hfill%
       $\square$\parfillskip=0pt\finalhyphendemerits=0\par}\goodbreak}
\def\endproof{\noproof\bigskip}

\newdimen\margin
\def\textno#1&#2\par{
   \margin=\hsize
   \advance\margin by -4\parindent
          \setbox1=\hbox{\sl#1}
   \ifdim\wd1 < \margin
      $$\box1\eqno#2$$
   \else
      \bigbreak
      \hbox to \hsize{\indent$\vcenter{\advance\hsize by -3\parindent
      \it\noindent#1}\hfil#2$}
      \bigbreak
   \fi}

\def\lateproof#1{\removelastskip\penalty55\medskip\noindent\setcounter{claim}{0}\setcounter{step}{0}{\bf Proof of #1. }} 

\begin{document}

\newcommand{\new}[1]{\textcolor{red}{#1}}
\def\COMMENT#1{}
\def\TASK#1{}
\newcommand{\APPENDIX}[1]{}
\newcommand{\NOTAPPENDIX}[1]{#1}
\renewcommand{\APPENDIX}[1]{#1}                    
\renewcommand{\NOTAPPENDIX}[1]{}                   
\newcommand{\todo}[1]{\begin{center}\textbf{to do:} #1 \end{center}}

\def\eps{{\varepsilon}}
\newcommand{\ex}{\mathbb{E}}
\newcommand{\pr}{\mathbb{P}}
\newcommand{\cB}{\mathcal{B}}
\newcommand{\cA}{\mathcal{A}}
\newcommand{\cE}{\mathcal{E}}
\newcommand{\cS}{\mathcal{S}}
\newcommand{\cF}{\mathcal{F}}
\newcommand{\cG}{\mathcal{G}}
\newcommand{\bL}{\mathbb{L}}
\newcommand{\bF}{\mathbb{F}}
\newcommand{\bZ}{\mathbb{Z}}
\newcommand{\cH}{\mathcal{H}}
\newcommand{\cC}{\mathcal{C}}
\newcommand{\cM}{\mathcal{M}}
\newcommand{\bN}{\mathbb{N}}
\newcommand{\bR}{\mathbb{R}}
\def\O{\mathcal{O}}
\newcommand{\cP}{\mathcal{P}}
\newcommand{\cQ}{\mathcal{Q}}
\newcommand{\cR}{\mathcal{R}}
\newcommand{\cJ}{\mathcal{J}}
\newcommand{\cL}{\mathcal{L}}
\newcommand{\cK}{\mathcal{K}}
\newcommand{\cD}{\mathcal{D}}
\newcommand{\cI}{\mathcal{I}}
\newcommand{\cV}{\mathcal{V}}
\newcommand{\cT}{\mathcal{T}}
\newcommand{\cU}{\mathcal{U}}
\newcommand{\cW}{\mathcal{W}}
\newcommand{\cX}{\mathcal{X}}
\newcommand{\cY}{\mathcal{Y}}
\newcommand{\cZ}{\mathcal{Z}}
\newcommand{\1}{{\bf 1}_{n\not\equiv \delta}}
\newcommand{\eul}{{\rm e}}
\newcommand{\Erd}{Erd\H{o}s}
\newcommand{\cupdot}{\mathbin{\mathaccent\cdot\cup}}
\newcommand{\whp}{whp }
\newcommand{\bX}{\mathcal{X}}
\newcommand{\bV}{\mathcal{V}}
\newcommand{\ordsubs}[2]{(#1)_{#2}}
\newcommand{\unordsubs}[2]{\binom{#1}{#2}}
\newcommand{\ordelement}[2]{\overrightarrow{\mathbf{#1}}\left({#2}\right)}
\newcommand{\ordered}[1]{\overrightarrow{\mathbf{#1}}}
\newcommand{\reversed}[1]{\overleftarrow{\mathbf{#1}}}
\newcommand{\weighting}[1]{\mathbf{#1}}
\newcommand{\weightel}[2]{\mathbf{#1}\left({#2}\right)}
\newcommand{\unord}[1]{\mathbf{#1}}
\newcommand{\ordscript}[2]{\ordered{{#1}}_{{#2}}}
\newcommand{\revscript}[2]{\reversed{{#1}}_{{#2}}}

\newcommand{\doublesquig}{%
  \mathrel{%
    \vcenter{\offinterlineskip
      \ialign{##\cr$\rightsquigarrow$\cr\noalign{\kern-1.5pt}$\rightsquigarrow$\cr}%
    }%
  }%
}

\newcommand{\defn}{\emph}

\newcommand\restrict[1]{\raisebox{-.5ex}{$|$}_{#1}}

\newcommand{\prob}[1]{\mathrm{\mathbb{P}}\left[#1\right]}
\newcommand{\probb}[1]{\mathrm{\mathbb{P}}_{b}\left[#1\right]}
\newcommand{\expn}[1]{\mathrm{\mathbb{E}}\left[#1\right]}
\newcommand{\expnb}[1]{\mathrm{\mathbb{E}}_{b}\left[#1\right]}
\newcommand{\probxj}[1]{\mathrm{\mathbb{P}}_{x(j)}\left[#1\right]}
\newcommand{\expnxj}[1]{\mathrm{\mathbb{E}}_{x(j)}\left[#1\right]}
\def\gnp{G_{n,p}}
\def\G{\mathcal{G}}
\def\lflr{\left\lfloor}
\def\rflr{\right\rfloor}
\def\lcl{\left\lceil}
\def\rcl{\right\rceil}

\newcommand{\qbinom}[2]{\binom{#1}{#2}_{\!q}}
\newcommand{\binomdim}[2]{\binom{#1}{#2}_{\!\dim}}

\newcommand{\grass}{\mathrm{Gr}}

\newcommand{\brackets}[1]{\left(#1\right)}
\def\sm{\setminus}
\newcommand{\Set}[1]{\{#1\}}
\newcommand{\set}[2]{\{#1\,:\;#2\}}
\newcommand{\krq}[2]{K^{(#1)}_{#2}}
\newcommand{\ind}[1]{$\mathbf{S}(#1)$}
\newcommand{\indcov}[1]{$(\#)_{#1}$}
\def\In{\subseteq}

\maketitle
\begin{abstract}
A tight Hamilton cycle in a $k$-uniform hypergraph ($k$-graph)~$G$ is a cyclic ordering of the vertices of~$G$ such that every set of~$k$ consecutive vertices in the ordering forms an edge. R\"{o}dl, Ruci\'{n}ski, and Szemer\'{e}di proved that for $k\geq 3$, every $k$-graph on~$n$ vertices with minimum codegree at least~$n/2+o(n)$ contains a tight Hamilton cycle. We show that the number of tight Hamilton cycles in such $k$-graphs is~$\exp(n\ln n-\Theta(n))$. As a corollary, we obtain a similar estimate on the number of Hamilton $\ell$-cycles in such $k$-graphs for all $\ell\in\{0,\dots,k-1\}$, which makes progress on a question of Ferber, Krivelevich and Sudakov.
\end{abstract}
\section{Introduction}
\subsection{Counting Hamilton cycles in graphs}
The problem of determining sufficient conditions for the existence of Hamilton cycles in graphs is one of the central topics in graph theory, and has given rise to extensive research.
A classical result of Dirac~\cite{D52} states that graphs on~$n\geq3$ vertices with minimum degree at least~$n/2$ (\textit{Dirac graphs}) contain a Hamilton cycle, and there are natural families of graphs which show that~$n/2$ is best possible.

Bollob\'{a}s~\cite{B95a} and Bondy~\cite{B95b} asked for an asymptotic estimate for the number of distinct Hamilton cycles in Dirac graphs.
In 2003, S\'{a}rk\"{o}zy, Selkow, and Szemer\'{e}di~\cite{SSS03} made substantial progress on this question by showing that $n$-vertex Dirac graphs contain~$\exp(n\ln n-\Theta(n))$ Hamilton cycles.
They also posed the question of whether this is the right order of magnitude for graphs satisfying other conditions known to ensure Hamiltonicity, like those of Ore, P\'{o}sa, and Chv\'{a}tal (see~\cite{B95b}).
Further, they conjectured that the minimum number of Hamilton cycles in $n$-vertex Dirac graphs is $\exp(n\ln n -n(1+\ln 2) -o(n))$.

Cuckler and Kahn~\cite{CK09b} analysed a self-avoiding random walk on the vertices of Dirac graphs to verify this conjecture as a consequence of a more precise result. 
Moreover, in a separate paper~\cite{CK09a}, they used entropy considerations to provide an upper bound for the number of Hamilton cycles in Dirac graphs.
More precisely, writing~$\Psi(G)$ to denote the number of distinct Hamilton cycles of a graph~$G$,
the main results of~\cite{CK09a} and~\cite{CK09b} together state that for any $n$-vertex Dirac graph~$G$, we have $\log_{2}\Psi(G)=2H(G)-n\log_{2}e-o(n)$, where~$H(G)$ is the \textit{entropy} of~$G$.
Combined with the result~\cite[Theorem 1.3]{CK09b} that Dirac graphs~$G$ on~$n$ vertices satisfy $H(G)\geq\frac{n}{2}\log_{2}\delta(G)$, this confirms the conjecture of~\cite{SSS03}.
Moreover, the parameter~$H(G)$ is the maximum of a concave function subject to linear constraints, and can thus be efficiently estimated.
This yields an efficient algorithm for estimating~$\Psi(G)$ for Dirac graphs~$G$, to within subexponential factors.
\subsection{Hamilton cycles in hypergraphs}
The study of Hamilton cycles in hypergraphs was initiated in a 1976 paper of Bermond, Germa, Heydemann, and Sotteau~\cite{B76}.
For $k$-uniform hypergraphs ($k$-\textit{graphs}), we may sensibly define a cycle in a number of ways (see for example~\cite{KO14, RR10, Z16}).
Let~$k\geq 2$ be an integer and let $\ell\in\{0,\dots,k-1\}$.
We say that a $k$-uniform hypergraph~$C$ is an $\ell$-cycle if there exists a cyclic ordering of the vertices of~$C$ such that every edge of~$C$ consists of~$k$ consecutive vertices and such that every pair of consecutive edges (in the natural ordering of the edges) intersects in precisely~$\ell$ vertices.
A \textit{Hamilton $\ell$-cycle} of a $k$-graph~$G$ is a subgraph $C\subseteq G$, where~$C$ is a $k$-uniform $\ell$-cycle with $V(C)=V(G)$.
Thus, if~$G$ contains a Hamilton $\ell$-cycle, then~$k-\ell$ divides~$|V(G)|$.
Moreover, if $\ell=0$ then a Hamilton $\ell$-cycle is just a perfect matching of~$G$.
We usually call a $(k-1)$-cycle a \textit{tight cycle}, and we say that a Hamilton $(k-1)$-cycle of a $k$-graph~$G$ is a \textit{tight Hamilton cycle} of~$G$.

We wish to generalise the study of Hamilton cycles in Dirac graphs to the setting of hypergraphs, and so we now need a natural hypergraph generalisation of the notion of degree.
Given a $k$-graph~$G$ and a set $S\subseteq V(G)$ of $k-1$ vertices, we say that the \textit{codegree} of~$S$ in~$G$, denoted~$d_{G}(S)$ (or simply~$d(S)$ when~$G$ is clear from the context), is the number of edges of~$G$ containing~$S$.
For a $k$-graph~$G$, we write~$\delta(G)$ for the minimum codegree over all $(k-1)$-sets $S\subseteq V(G)$, and refer to this quantity as the \textit{minimum codegree} of~$G$.

Katona and Kierstead~\cite{KK99} gave a sufficient condition on the minimum codegree for $k$-graphs to have a tight Hamilton cycle.
Further, they conjectured that for all integers~$k\geq2$, a minimum codegree of at least~$n/2$ suffices for $n$-vertex $k$-graphs.
R\"{o}dl, Ruci\'{n}ski, and Szemer\'{e}di proved an asymptotic version~\cite{RRS06} of the~$k=3$ case of this conjecture, and then an exact version for large~$n$~\cite{RRS11}.
The work of~\cite{RRS06} was shortly afterwards generalised to all integers~$k\geq 3$ by the same authors~\cite{RRS08}.
Further results on tight Hamilton cycles can be found e.g. in~\cite{AL, RRRSS19}.
For $(k-\ell)\nmid k$, K\"{u}hn, Mycroft, and Osthus~\cite{KMO10} asymptotically determined the threshold for the existence of a Hamilton $\ell$-cycle (this generalised previous results in~\cite{KO06b,KKMO11,HS10}).
Subsequently several exact results were proved in~\cite{CM14,HZ15}.
It turns out that the threshold is significantly below~$n/2$ if $(k-\ell)\nmid k$.
For all other cases it follows from the result of~\cite{RRS08} that the threshold is asymptotically~$n/2$.
\subsection{Our main result}
Ferber, Krivelevich, and Sudakov~\cite{FKS16} were the first to generalise the study of counting Hamilton cycles to the hypergraph setting (and also considered perfect matchings). They proved for $1\leq\ell < k/2$ that if a $k$-graph $G$ on $n$ vertices with $(k-\ell)\mid n$ satisfies $\delta(G)\geq \alpha n$ for some $\alpha>1/2$, then $G$ contains $(1-o(1))^{n}\cdot n! \cdot\left(\frac{\alpha}{\ell!(k-2\ell)!}\right)^{\frac{n}{k-\ell}}$ Hamilton $\ell$-cycles. 
As a natural question, they asked whether this can be generalized to all $\ell$. 

We adapt some ideas from the random walk analysis of~\cite{CK09b} to show that any large $k$-graph whose minimum codegree is slightly above~$n/2$ contains a large number of tight Hamilton cycles.
\begin{theorem}\label{hamcycs}
For a fixed integer $k\geq 2$ and a fixed constant~$\gamma>0$, the number of tight Hamilton cycles of a $k$-graph~$G$ on~$n$ vertices with $\delta(G)\geq (1/2+\gamma)n$ is $\exp(n\ln n-\Theta(n))$.
\end{theorem}
Notice that we claim this number of tight Hamilton cycles holds with equality, up to the exponential error bound~$\exp(-\Theta(n))$.
We discuss this error bound further in the concluding remarks.
It will suffice to show that the lower bound holds, since any $k$-graph on~$n$ vertices trivially has at most~$(n-1)!/2$ distinct tight Hamilton cycles.
Theorem~\ref{hamcycs} easily yields the following corollary about the number of Hamilton $\ell$-cycles in a $k$-graph whose codegrees are slightly above~$n/2$, for each $\ell\in\{0,\dots,k-1\}$.
\begin{cor}\label{hamellcycles}
For a fixed integer $k\geq2$ and a fixed constant $\gamma>0$, the number of Hamilton $\ell$-cycles of a $k$-graph~$G$ on~$n$ vertices with $(k-\ell)\mid n$ and $\delta(G)\geq (1/2+\gamma)n$ is
\begin{enumerate}[label=\upshape(\roman*)]
    \item $\exp\left(\left(1-\frac{1}{k}\right)n\ln n -\Theta(n)\right)$, if $\ell=0$;
    \item $\exp(n\ln n -\Theta(n))$, if $\ell\in[k-1]$.
\end{enumerate}
\end{cor}

This addresses the above mentioned question of Ferber, Krivelevich and Sudakov (though our result is less precise than theirs for $\ell<k/2$).
We remark that the $-\frac{1}{k}$-term for the case of perfect matchings is missing in \cite[Theorem 1.1]{FKS16}, but follows from their proof.
Finally, recall that the minimum codegree threshold for the existence of Hamilton $\ell$-cycles can be below $n/2$ when $\ell<k-1$. It would thus be a natural question to extend the counting results to this larger range.
For the rest of the paper, we focus on counting tight Hamilton cycles.
\section{Sketch of the proof of Theorem \ref{hamcycs}}\label{Sketch}
In this section we provide a rough sketch of the proof of our main result.
\subsection{Basic notation}
We first need to introduce some notation that we use throughout the paper.
For a set~$V$ and a natural number~$\ell$, we write~$\unordsubs{V}{\ell}$ to denote the set of all unordered $\ell$-subsets of distinct elements of~$V$.
We write~$\ordsubs{V}{\ell}$ to denote the set of all ordered $\ell$-subsets of distinct elements of~$V$, so that $\left|\ordsubs{V}{\ell}\right| = \ell!\big|\unordsubs{V}{\ell}\big|$.
We usually use boldface capital letters to denote unordered subsets~$\unord{S}\in\unordsubs{V}{\ell}$ of the fixed size~$\ell$, and we exclusively use boldface capital letters with arrows above to denote ordered subsets~$\ordered{S}\in\ordsubs{V}{\ell}$.
When an ordered tuple~$\ordered{S}\in\ordsubs{V}{\ell}$ is first given, the arrow will exclusively point to the right. We may subsequently drop the arrow to denote the unordered version of this~$\ell$-set,
so that if~$\ordered{S}$ is the ordered sequence of~$\ell$ distinct elements $(x_{1},\dots,x_{\ell})$, then~$\unord{S}$ subsequently used without the arrow denotes the unordered set $\{x_{1},\dots,x_{\ell}\}$.
Moreover, we write~$\reversed{S}$ to denote the ordered $\ell$-tuple obtained by reversing the ordering of~$\ordered{S}$,
so that $\reversed{S}=(x_{\ell},x_{\ell-1},\dots,x_{1})$.
Let~$G=(V,E)$ be a hypergraph and let~$U\subseteq V(G)$.
Then the sub(hyper)graph of~$G$ \textit{induced} by~$U$, denoted~$G[U]$, is the hypergraph~$H=(V(H), E(H))$, where~$V(H)=U$, and~$E(H)$ is precisely the set of all edges of~$G$ containing only vertices in~$U$.
We write~$G-U$ to denote the hypergraph~$G'\subseteq G$ obtained from~$G$ by deleting the vertices in~$U$ and all edges of~$G$ containing any vertex in~$U$.
We say that a $k$-graph~$P$ is a $k$-\textit{uniform tight path} (or simply \textit{tight path} if~$k$ is clear from the context) if~$P$ admits an ordering of its vertices $V(P)=\{v_{1},\dots,v_{m}\}$ such that $E(P)=\{\{v_{i},\dots,v_{i+k-1}\}\colon1\leq i \leq m-(k-1)\}$.
The \textit{ends} of~$P$ are the ordered $(k-1)$-tuples $(v_{1},\dots,v_{k-1})$ and $(v_{m},\dots,v_{m-k+2})$.
We also say that~$P$ \textit{connects} the ends of~$P$.
We say that a tight path~$P$ with~$m$ edges (and thus with $m+(k-1)$ vertices) is an~$m$-\textit{path}, and has \textit{length}~$m$.
For a $k$-graph~$G$ and an integer~$t\geq k$ we say that a sequence $(v_{1}, v_{2}, \dots, v_{t})$ of (not necessarily distinct) vertices is a \textit{walk} in~$G$ if every set of~$k$ consecutive vertices in the sequence forms an edge.
Let~$\gamma>0$ be a constant.
A $k$-graph~$G$ on~$n$ vertices is called~$\gamma$-\textit{Dirac} if $\delta(G) \geq (1/2 + \gamma)n$.
Finally, given a hypergraph~$G$, we say a weighting of the edges $\weighting{x}\colon E(G)\rightarrow\bR^{+}$ is a \textit{fractional matching} if we have $\sum_{e\ni v}\weightel{x}{e}\leq1$ for every~$v\in V(G)$, and we say that~$\weighting{x}$ is \textit{perfect} if $\sum_{e\ni v}\weightel{x}{e} =1$ for every~$v\in V(G)$.
\subsection{Outline of the argument}
Let~$\gamma>0$, and let~$G$ be an $n$-vertex $k$-graph satisfying $\delta(G)\geq (1/2+\gamma)n$, where~$k\geq 2$ and~$n$ is sufficiently large.
The main step of our proof is to count tight paths of length $n-o(n)$ in~$G$.
Using the framework of R\"{o}dl, Ruci\'{n}ski, and Szemer\'{e}di~\cite{RRS08}, which is based on the absorption technique, we can complete each such long path into a tight Hamilton cycle of~$G$.
The key lemma (Lemma \ref{iteration}) in the proof of Theorem~\ref{hamcycs} states that we can find many paths of length~$\sqrt{n}$ in~$G$,
all starting at the same ordered~$(k-1)$-tuple $\ordered{S}\in\ordsubs{V(G)}{k-1}$, such that for each such path the remainder of~$G$ still has minimum codegree at least~$\left(\frac{1}{2}+\gamma-n^{-2/3}\right)(n-\sqrt{n})$.
The proof of this `iteration lemma' is the sole focus of Section~\ref{countingshort}, and involves the analysis of a self-avoiding random walk~$\cX$ on the vertices of~$G$.
In order to prove the iteration lemma, we first need to show that~$G$ admits a perfect fractional matching which is `normal', which means that each edge of~$G$ has weight~$\Theta(n^{-k+1})$.
We construct such a normal perfect fractional matching~$\weighting{x}$ in Section~\ref{fractional} via a probabilistic argument based on switchings (it is not clear how to generalise the entropy-based approach of~\cite{CK09b} to the hypergraph setting).

In Section~\ref{countingshort}, we use~$\weighting{x}$ to define the transition probabilities of the random walk~$\cX$.
We construct~$\cX$ such that an outcome of~$\cX$ corresponds to a tight path in~$G$ of length~$\sqrt{n}$ which starts at some given~$\ordered{S}\in\ordsubs{V(G)}{k-1}$.
We wish to count the number of outcomes of~$\cX$ which essentially leave the~$\gamma$-Dirac property of the remaining graph intact.
Such outcomes of~$\cX$ are called \textit{good} walks.
It will suffice to show that~$\cX$ is good with probability at least~$1/2$.
To do this, we will show that it is likely that the vertices that~$\cX$ visits look roughly like a uniformly random subset of the vertices of~$G$, of appropriate size.
We will show that the behaviour of~$\cX$ over a small number of steps can be assumed to be very close to the behaviour of a modified version of~$\cX$, in which the walk is allowed to revisit vertices.
We use the normality property of~$\weighting{x}$ to show that the modified walk mixes rapidly, and we use the fact that~$\weighting{x}$ is a perfect fractional matching to show that, under the stationary distribution, each vertex is essentially visited with the same probability.
Thus, roughly speaking, the distribution of the vertices for~$\cX$ to visit at any step is close to uniform on~$V(G)$.
We give a more thorough sketch of the proof of the iteration lemma in Section~\ref{propersketch}.

In Section~\ref{countinglong}, we focus on repeatedly applying the iteration lemma to obtain many long paths in~$G$.
Let~$P$ be a~$\sqrt{n}$-path in~$G$ obtained from the first iteration of the iteration lemma, let~$\unord{T}$ be the unordered set consisting of the final~$k-1$ vertices of~$P$, and let $G_{P}\coloneqq G-(V(P)\setminus\unord{T})$.
The idea is that, since the~$\gamma$-Dirac property is essentially intact in $G_{P}$,
we may find a new normal perfect fractional matching $\weighting{x}_{P}\colon E(G_{P})\rightarrow \bR^{+}$ and can thus apply the iteration lemma to~$G_{P}$.
In this second iteration, we insist that all the walks~$\cX$ start at the final ordered $(k-1)$-tuple of~$P$.
Then we may attach any of the paths~$P'$ from the second iteration onto~$P$ to obtain a longer tight path in~$G$ which still leaves the~$\gamma$-Dirac property of the remaining graph essentially intact.
We show that we may iterate this process until fewer than~$n^{7/8}$ vertices of~$G$ remain, and we multiplicatively use the count of paths given by the iteration lemma to deduce that the number of resulting long paths of~$G$ is essentially the number given in the statement of Theorem~\ref{hamcycs}.
(Observe that each combination of paths yields a different concatenated path).
Finally then, we complete the proof of Theorem~\ref{hamcycs} by absorbing the vertices left over by each such long path into a tight Hamilton cycle of~$G$.
\section{Preliminaries}\label{prelims}
In the following section, we collect further notation, as well as some results that we will use throughout the paper.
\subsection{Notation}
Let~$\gamma>0$ be a constant.
We say that a $k$-graph~$G$ is an $(n,k,\gamma)$-\textit{graph} if~$G$ has~$n$ vertices and~$G$ is~$\gamma$-Dirac.
When~$G$ is clear from the context, we often write~$V$ instead of~$V(G)$.
For a $k$-graph~$G$ and $\unord{S}\in\unordsubs{V}{k-1}$, we write $N_{G}(\unord{S})\coloneqq\{v\in V\colon \unord{S}\cup\{v\}\in E(G)\}$.
We say that~$\unord{S}$ is \textit{isolated} if $N_{G}(\unord{S})=\emptyset$, and that~$\unord{S}$ is \textit{non-isolated} if~$\unord{S}$ is not isolated.
For a positive integer~$\ell$, we say that a walk $(v_{1},\dots,v_{\ell+k-1})$ on the vertices of~$G$ is an $\ell$-\textit{walk}.
Let $\ordered{S},\ordered{T}\in\ordsubs{V}{k-1}$.
We say an~$\ell$-walk $(v_{1},\dots,v_{\ell+k-1})$ in~$G$ is an $\ell$-\textit{walk from}~$\ordered{S}$ \textit{to}~$\ordered{T}$ if $\ordered{S}=(v_{1},\dots,v_{k-1})$ and $\ordered{T}=(v_{\ell+1},\dots,v_{\ell+k-1})$.
A \textit{matching}~$M$ of a $k$-graph~$G$ is a collection of vertex-disjoint edges of~$G$, and we say that~$M$ is \textit{perfect} if every vertex $v\in V(G)$ is included in some edge of~$M$.
Let~$M$ be a matching in a~$k$-graph~$G$. Where it has no effect on the argument, we sometimes abuse notation and identify~$M$ with the subgraph~$M'\subseteq G$ satisfying $E(M')=M$ and $V(M')=\bigcup_{e\in M}e$.
For finite sets~$U\subseteq V$ and a function $f:V\rightarrow\bR$, we define $f(U)\coloneqq\sum_{u\in U}f(u)$, and $||f||_{\infty}\coloneqq\max_{v\in V}f(v)$.
We write $\mathds{1}_{U}:V\rightarrow\{0,1\}$ to be the \textit{indicator function} for~$U$, defined by $\mathds{1}_{U}(x)=1$ if~$x\in U$, and $\mathds{1}_{U}(x)=0$ otherwise.
For an event~$\cE$ in a probability space, we write~$\cE^{c}$ to denote the complement of~$\cE$.
We write~$\log x$ to mean~$\log_{2}x$, and we write~$\ln x$ to mean~$\log_{e}x$.
We also write $a=(1\pm b)c$ to mean $(1-b)c < a < (1+b)c$.
For a natural number~$n$ we write $[n]\coloneqq\{1,\dots,n\}$.
We write $x\ll y$ to mean that for any $y\in (0,1]$ there exists an $x_{0}\in (0,1)$ such that for all $0<x\leq x_{0}$ the subsequent statement holds.
Hierarchies with more constants are defined similarly and should be read from the right to the left.
Constants in hierarchies will always be real numbers in $(0,1]$.
Moreover, if $1/x$ appears in a hierarchy, this implicitly means that $x$ is a natural number.
More precisely, $1/x\ll y$ means that for any $y\in (0,1]$, there exists an $x_{0}\in\bN$ such that for all $x\in\bN$ with $x\geq x_{0}$ the subsequent statement holds.
We assume large numbers to be integers if this does not affect the argument.
\subsection{Probabilistic tools}
In this subsection we collect some probabilistic definitions and results that we will need throughout the paper.

The \textit{total variation distance} between two probability measures~$\mu$ and~$\nu$ on a finite set~$S$ is
$d_{TV}(\mu, \nu) := \sup\{|\mu(T)-\nu(T)|\colon T\subseteq S\}$. It is well-known that the total variation distance satisfies
\begin{equation}\label{eq:variation2}
d_{TV}(\mu, \nu)  = \frac{1}{2}\sum_{s \in S}\left|\mu(s)-\nu(s)\right|=\inf\{\prob{X \neq Y}\},
\end{equation}
where the infimum is taken over coupled random variables~$X$ and~$Y$ having laws~$\mu$ and~$\nu$ respectively (see \cite[p.119]{D10} for more details).
We write~$d_{TV}(X, Y) $ for the total variation distance between the laws of the random variables~$X$ and~$Y$.

Next, we need an inequality of~\cite{CK09b}, which follows easily from Azuma's inequality.
\begin{lemma}[{\cite[Lemma 5.3]{CK09b}}]\label{azuma}
Let $X_{0}, X_{1}, ...$ be random variables taking values in a set~$V$, and let $g\colon V \rightarrow \mathds{R}$.
Then for any~$t>0$ and any $p, q\in\bN$, we have
\[
\prob{\left|\sum_{k=0}^{p}\left(g\left(X_{k+q}\right) - \expn{g\left(X_{k+q}\right) | X_{0}, ..., X_{k}}\right)\right|> t||g||_{\infty}\sqrt{pq}} < 2qe^{-t^{2}/2}.
\]
\end{lemma}
We will need the following Chernoff-type bound (see~\cite{CL06} and~\cite{JLR00} for example).
\begin{lemma}\label{chernoff}
Let~$X$ be a random variable with a binomial or hypergeometric distribution. Suppose~$\ex[X]>0$ and let~$t>0$. Then $\prob{X\leq\expn{X}-t} \leq e^{-t^{2}/(2\expn{X})}$.
\end{lemma}
We conclude this section with a result which shows that most small sets in an $(n,k,\gamma)$-graph inherit the Dirac condition. \APPENDIX{The proof is given in the appendix.}\NOTAPPENDIX{We omit the proof; for details, see the appendix of the arXiv version of the paper.}
\begin{prop}\label{tidytool}
Let $1/n\ll 1/m \ll\gamma, 1/k, 1/t, 1/\ell$, where~$\ell\mid n$, and let~$G$ be an $(n,k,\gamma)$-graph.
Let~$\cP$ be a partition of~$V$ into $\ell$-sets and let $\cP_{0}\subseteq\cP$ be of size $|\cP_{0}|= t$.
Pick $\cP'\subseteq\cP\setminus\cP_{0}$ of size $m$ uniformly at random.
Then $\prob{G[\bigcup(\cP_{0}\cup\cP')]\ \text{is}\ \gamma/2\text{-Dirac}}\geq1-e^{-\sqrt{m}}$.
\end{prop}
\subsection{Tight Hamilton-connectedness}
Let~$G$ be an $(n,k,\gamma)$-graph and let~$P$ be a tight path in~$G$.
We say that~$P$ is a \textit{tight Hamilton path} of~$G$ if~$V(P)=V$.
We say that~$G$ is \textit{tight Hamilton-connected} if for any disjoint $\ordered{S}, \ordered{T}\in\ordsubs{V}{k-1}$, there is a tight Hamilton path of~$G$ which connects~$\ordered{S}$ and~$\ordered{T}$.
We will deduce from the results in~\cite{RRS08} that large $(n,k,\gamma)$-graphs are tight Hamilton-connected for~$k\geq 3$.
This will be important in the absorption step of our main argument, and also in the mixing part of our random walk analysis.
We begin by stating the main theorem of~\cite{RRS08}.
\begin{theorem}[{\cite[Theorem 1.1]{RRS08}}]\label{RRSmainthm}
Let $1/n\ll\gamma,1/k$, where $k\geq 3$, and let~$G$ be an $(n,k,\gamma)$-graph.
Then~$G$ contains a tight Hamilton cycle.
\end{theorem}
The next lemma ensures the existence of an `absorbing path'~$A$, which can absorb small sets of vertices into its interior.
\begin{lemma}[{\cite[Lemma 2.1]{RRS08}}]\label{absorbingpath}
Let $1/n \ll \gamma, 1/k$, where $k\geq 3$, suppose that $\gamma\leq 1/(32k)$, set $\beta\coloneqq 2^{k-4}\gamma^{2k}n$, and let~$G$ be an $(n,k,\gamma)$-graph.
Then there exists a tight path~$A$ in~$G$ with $|V(A)|\leq 16k\gamma^{k-1}n$ such that for every subset $U\subseteq V\setminus V(A)$ of size~$|U|\leq\beta$, there is a tight path~$A_{U}$ in~$G$ with $V(A_{U}) = V(A)\cup U$ and such that~$A_{U}$ has the same ends as~$A$.
\end{lemma}
The next lemma will enable us to find constant-length tight paths between any disjoint pair of ordered $(k-1)$-sets of vertices.
\begin{lemma}[{\cite[Lemma 2.4]{RRS08}}]\label{connectinglemma}
Let $1/n\ll \gamma, 1/k$, where $k\geq 3$, and let $G$ be an $(n,k,\gamma)$-graph. Then for every $\ordered{S},\ordered{T}\in\ordsubs{V}{k-1}$ with $\unord{S}\cap \unord{T} = \emptyset$, there is an $\ell$-path $P$ in $G$ with $\ell\leq 2k/\gamma^{2}$ that connects $\ordered{S}$ and $\ordered{T}$.
\end{lemma}
We are now ready to prove that large $(n,k,\gamma)$-graphs are tight Hamilton-connected.
\begin{lemma}\label{hampath}
Let $1/n \ll \gamma,1/k$, where~$k\geq2$, and let $G$ be an $(n,k,\gamma)$-graph.
Then~$G$ is tight Hamilton-connected.
\end{lemma}
\begin{proof}
Firstly, note that this result follows easily from Dirac's Theorem\COMMENT{Let $x,y \in V(G)$ and consider $G'\coloneqq G-\{x,y\}$. By Dirac's Theorem,~$G'$ has a Hamilton cycle~$C$. Now we use the degree condition in~$G$ to see that the set~$N(x)$ intersects the set~$N^{+}(y)$ on~$V(C)$, where~$N^{+}(y)$ is the set of successors of neighbours of~$y$ on~$C$.} for the case $k=2$. Now, suppose $k\geq 3$ and suppose without loss of generality that $\gamma>0$ is sufficiently small in comparison to~$k$.
Let~$\ordered{S}, \ordered{T}\in\ordsubs{V}{k-1}$ be disjoint, and write $\ordered{S}=(s_{1},\dots,s_{k-1})$.
Set $\gamma'\coloneqq 3\gamma/4$, so that $G'\coloneqq G-(\unord{S}\cup\unord{T})$ is $\gamma'$-Dirac, and set $n'\coloneqq n-2(k-1)$.
We apply Lemma~\ref{absorbingpath} to~$G'$ to obtain a tight path~$A$ in~$G'$ with $|V(A)|\leq 16k(\gamma')^{k-1}n'$, with the properties as stated in Lemma~\ref{absorbingpath}.
Choose a set~$W\subseteq V\setminus(\unord{S}\cup\unord{T}\cup V(A))$ of size $(\gamma')^{3k}n'$ uniformly at random among all sets of that size.
Then a simple application\COMMENT{For each~$\unord{M}\in\unordsubs{V}{k-1}$, put $X_{\unord{M}}\coloneqq |N_{G}(\unord{M})\cap W|$.
Then $X_{\unord{M}}\sim\text{hyp}(|W|,d_{G-(\unord{S}\cup\unord{T}\cup V(A))}(\unord{M}), n-2(k-1)-|V(A)|)$.
So $\expn{X_{\unord{M}}}\geq (1/2+3\gamma/5)|W|$.
Then $\prob{X_{\unord{M}}\leq(1/2+\gamma/2)|W|}\leq\exp(-\gamma^{2}|W|/200)=\exp(-\Theta(n))$.
Taking a union bound over all $(k-1)$-tuples of~$G$, we see that the probability of any such tuple having codegree in~$W$ less than $(1/2+\gamma/2)|W|$ is at most~$1/2$.}
of Lemma~\ref{chernoff} shows that with high probability, for all $\unord{M}\in\unordsubs{V}{k-1}$ we have $|N_{G}(\unord{M})\cap W|\geq (1/2+\gamma/2)|W|$, and thus in particular,~$G[W]$ is $\gamma/2$-Dirac.
We fix such a choice of~$W$.
Set $G''\coloneqq G-(\unord{S}\cup\unord{T}\cup V(A) \cup W)$, and notice that~$G''$ is $\gamma/2$-Dirac.

We apply Theorem~\ref{RRSmainthm} to~$G''$ to obtain a Hamilton cycle~$C$ of~$G''$.
Delete~$k-1$ consecutive edges of~$C$ to obtain a Hamilton path~$P$ of~$G''$ with ends~$\ordered{X}$ and~$\ordered{Y}$.
We use the property that all $(k-1)$-tuples in~$G$ have high codegree in~$W$ to find, for each $i\in [k-1]$ in turn, a vertex $v_{i}\in W\setminus\{v_{1},\dots,v_{i-1}\}$ such that $\{s_{i},\dots,s_{k-1},v_{1},\dots,v_{i}\}$ is an edge.
Let $\ordered{S'}\coloneqq(v_{k-1},\dots,v_{1})$, so that we have found a $(k-1)$-path~$P_{S}$ with ends~$\ordered{S}$ and~$\ordered{S'}$.
Let~$\ordscript{A}{1}$ and~$\ordscript{A}{2}$ be the ends of~$A$.
We similarly find mutually disjoint $\ordered{X'}, \ordered{Y'}, \ordered{A'_{1}}, \ordered{A'_{2}}, \ordered{T'}\in\ordsubs{W}{k-1}$ and $(k-1)$-paths $P_{X}, P_{Y}, P_{A_{1}}, P_{A_{2}}, P_{T}$ with the corresponding pairs of ends.
Since~$G[W]$ is $\gamma/2$-Dirac, we can apply Lemma~\ref{connectinglemma} to obtain a path~$P_{SX}$ of length at most $8k/\gamma^{2}$ in~$G[W]$ which connects~$\reversed{S'}$ and~$\reversed{X'}$.
Since~$P_{SX}$ contains so few vertices, we can repeat the process to find disjoint paths~$P_{YA_{1}}$ and~$P_{A_{2}T}$ in~$G[W]$ with ends~$\reversed{Y'}$ and~$\reversed{A'_{1}}$, and~$\reversed{A'_{2}}$ and~$\reversed{T'}$, respectively.
Let $W'\coloneqq W\setminus (V(P_{SX})\cup V(P_{YA_{1}})\cup V(P_{A_{2}T}))$.
We apply the absorbing property of~$A$ to obtain a tight path~$A_{W'}$ in~$G$ with~$V(A_{W'})=V(A)\cup W'$, such that~$A_{W'}$ has ends~$\ordered{A_{1}}$ and~$\ordered{A_{2}}$.

Then $P_{S}\cup P_{SX}\cup P_{X} \cup P\cup P_{Y}\cup P_{YA_{1}}\cup P_{A_{1}} \cup A_{W'}\cup P_{A_{2}} \cup P_{A_{2}T}\cup P_{T}$ is a tight Hamilton path in~$G$ which connects~$\ordered{S}$ and~$\ordered{T}$.
\end{proof}
\section{Normal perfect fractional matchings}\label{fractional}
Let $k\geq2$ and let~$G$ be a $k$-graph on~$n$ vertices. We say that an edge weighting $\weighting{x}\colon E(G)\rightarrow\bR^{+}$ is $C$-\textit{normal} if \begin{equation}\label{eq:doublebound}
\frac{1}{Cn^{k-1}}\leq\weightel{x}{e}\leq\frac{C}{n^{k-1}}, \hspace{7mm}\text{for each}\hspace{1.5mm}e\in E(G).
\end{equation}
In this section we adapt some ideas of~\cite{CKPY} to show that an $(n,k,\gamma)$-graph~$G$ admits a normal perfect fractional matching (see Lemma~\ref{normality}).
This will be an essential tool in our random walk analysis for showing that the random walk is roughly equally likely to visit any vertex.
The idea is to construct a perfect fractional matching of~$G$ in which the weight of any edge~$e$ is set to be the probability that~$e$ is included in a uniformly random perfect matching of~$G$.
(A sufficiently large $(n,k,\gamma)$-graph with~$k\mid n$ has at least one perfect matching~\cite{KO06,RRS09}).
A crucial feature of this approach is that any edge~$e$ is roughly equally likely to be included in a uniformly random perfect matching of~$G$.
We show this using the so-called `switching method' in a similar way as in~\cite{CKPY}.
Let~$k\geq2$, let~$G$ be a $k$-graph, let~$e\in E(G)$, and let~$M_{\ell}$ be a perfect matching of~$G$ containing precisely~$\ell$ edges intersecting~$e$.
Supposing $0\leq\ell\leq k-1$, we define an $(e,M_{\ell})$\textit{-upswitching} to be a matching~$Y$ of~$G$ satisfying
\begin{enumerate}[label=\upshape(\roman*)]
\item $e\subseteq V(Y)$;
\item $Y$ contains precisely $\ell+1$ edges intersecting~$e$;
\item for all~$e'\in M_{\ell}$, we have either~$e'\subseteq V(Y)$ or $e'\cap V(Y)=\emptyset$.
\end{enumerate}
Supposing instead that $\ell\in [k]$, we define an $(e,M_{\ell})$\textit{-downswitching} to be a matching~$Y$ of~$G$ satisfying
\begin{enumerate}[label=\upshape(\roman*)]
\item $e\subseteq V(Y)$;
\item $Y$ contains precisely $\ell-1$ edges intersecting~$e$;
\item for all~$e'\in M_{\ell}$, we have either~$e'\subseteq V(Y)$ or $e'\cap V(Y)=\emptyset$.
\end{enumerate}
Note that if~$Y$ is an $(e, M_{\ell})$-upswitching, then we can obtain a new perfect matching~$M'$ from~$M_{\ell}$ by replacing~$M_{\ell}[V(Y)]$ with~$Y$.
Then~$M'$ contains exactly~$\ell+1$ edges intersecting~$e$.
Similarly, if~$Y$ is an $(e, M_{\ell})$-downswitching, then~$M'$ has exactly~$\ell-1$ edges intersecting~$e$.
\begin{lemma}\label{perfect}
Let $1/n \ll 1/C \ll \gamma, 1/k$, where~$k \geq 2$ and~$k\mid n$.
Let~$G$ be an $(n,k,\gamma)$-graph, and let~$M$ be a uniformly random perfect matching of~$G$.
Then for each~$e\in E(G)$, we have
\begin{equation}\label{eq:theclaim}
\frac{1}{Cn^{k-1}}\leq\prob{e\in M}\leq\frac{C}{n^{k-1}}.
\end{equation}
\end{lemma}
\begin{proof}
Choose new integers~$m$ and~$B$ satisfying $1/n\ll1/C\ll1/B\ll1/m\ll\gamma,1/k$, and fix~$e\in E(G)$.
For each integer $\ell\in [k]$, let~$\cM_{\ell}$ be the set of perfect matchings of~$G$ containing precisely~$\ell$ edges intersecting~$e$.
Note that $\prob{e\in M}=|\cM_{1}|/(|\cM_{1}|+\dots+|\cM_{k}|)$ and recall that there is at least one perfect matching of~$G$ since~$G$ is $\gamma$-Dirac (so the denominator here is nonzero).
We first bound $|\cM_{\ell}|/|\cM_{\ell+1}|$ from above and below for each $\ell\in [k-1]$, and~(\ref{eq:theclaim}) will follow quickly.
Let~$\ell\in [k-1]$.
We define an auxiliary bipartite multigraph~$G_{e,\ell}^{\uparrow}$ with vertex bipartition~$(\cM_{\ell}, \cM_{\ell+1})$.
For each~$M_{\ell}\in\cM_{\ell}$ and each $(e, M_{\ell})$-upswitching~$Y$ of \textit{size}~$m$ (containing precisely~$m$ edges), we add an edge in~$G_{e,\ell}^{\uparrow}$ from~$M_{\ell}$ to the matching $M_{\ell+1}\in\cM_{\ell+1}$ obtained by replacing~$M_{\ell}[V(Y)]$ with~$Y$.
Write~$\delta_{\cM_{\ell}}^{e,\uparrow}$ to denote the minimum degree in~$G_{e,\ell}^{\uparrow}$ over all~$M_{\ell}\in\cM_{\ell}$,
and write~$\Delta_{\cM_{\ell+1}}^{e,\uparrow}$ to denote the maximum degree in~$G_{e,\ell}^{\uparrow}$ over all~$M_{\ell+1}\in\cM_{\ell+1}$.
By double-counting~$|E(G_{e,\ell}^{\uparrow})|$, we obtain $|\cM_{\ell}|/|\cM_{\ell+1}|\leq\Delta_{\cM_{\ell+1}}^{e,\uparrow}/\delta_{\cM_{\ell}}^{e,\uparrow}$.
To bound~$\Delta_{\cM_{\ell+1}}^{e,\uparrow}$, we fix~$M_{\ell+1}\in\cM_{\ell+1}$ and bound the number of pairs~$(M_{\ell}, Y)$, where~$M_{\ell}\in\cM_{\ell}$ and~$Y$ is an $(e, M_{\ell})$-upswitching of size~$m$ that produces~$M_{\ell+1}$.
Note that any such~$Y$ must contain all vertices in the~$\ell+1$ edges of~$M_{\ell+1}$ intersecting~$e$, and there are at most~$n^{m-\ell-1}$ choices for the other~$m-\ell-1$ edges of~$M_{\ell+1}$ whose vertices to include in~$V(Y)$.
Once~$V(Y)$ is fixed, there are at most~$(mk)!$ choices for~$M_{\ell}[V(Y)]$ (and hence for~$M_{\ell}$).
Thus, we have~$\Delta_{\cM_{\ell+1}}^{e,\uparrow}\leq (mk)!n^{m-\ell-1}$.

To bound~$\delta_{\cM_{\ell}}^{e,\uparrow}$, we fix $M_{\ell}\in\cM_{\ell}$ and bound the number of $(e, M_{\ell})$-upswitchings of size~$m$ from below.
Let $U(M_{\ell})\coloneqq\{e'\in M_{\ell}\colon e\cap e'\neq\emptyset\}$.
Note that any $(e, M_{\ell})$-upswitching~$Y$ of size~$m$ must include all the vertices in~$U(M_{\ell})$, and there are~$\binom{n/k-\ell}{m-\ell}$ choices for the remaining~$m-\ell$ edges of~$M_{\ell}$ whose vertices to include in~$V(Y)$.
We apply Proposition~\ref{tidytool} (with $\cP=M_{\ell}$, $\cP_{0}=U(M_{\ell})$, and with~$m-\ell$,~$k$,~$\ell$ playing the roles of~$m$,~$\ell$,~$t$, respectively) to deduce that there are at least $(1-e^{-\sqrt{m-\ell}})\binom{n/k-\ell}{m-\ell}\geq (mk)^{-m}n^{m-\ell}$ choices of $X\subseteq M_{\ell}\setminus U(M_{\ell})$ of size~$m-\ell$ such that~$G[V(X\cup U(M_{\ell}))]$ is $\gamma/2$-Dirac.
Note that for each such~$X$, we may first choose a matching~$U'$ of size~$\ell+1$ in~$G[V(X\cup U(M_{\ell}))]$ such that $e\subseteq V(U')$ and~$e$ intersects every edge in~$U'$, and then choose a perfect matching~$Y'$ of~$G[V(X\cup U(M_{\ell}))\setminus V(U')]$. Then $Y\coloneqq Y'\cup U'$ is an $(e, M_{\ell})$-upswitching of size~$m$, unique to this choice\COMMENT{Other choices of~$X$ have different vertex sets.}
of~$X$.
We deduce that $\delta_{\cM_{\ell}}^{e,\uparrow}\geq (mk)^{-m}n^{m-\ell}$, and conclude that $|\cM_{\ell}|/|\cM_{\ell+1}|\leq (mk)!(mk)^{m}/n\leq B/n$.

We now bound the terms~$|\cM_{\ell}|/|\cM_{\ell+1}|$ from below analogously.
Let $\ell\in [k-1]$.
We define an auxiliary bipartite multigraph~$G_{e,\ell+1}^{\downarrow}$ with vertex bipartition~$(\cM_{\ell}, \cM_{\ell+1})$.
For each~$M_{\ell+1}\in\cM_{\ell+1}$ and each $(e, M_{\ell+1})$-downswitching~$Y$ of size~$m$, we add an edge in~$G_{e,\ell+1}^{\downarrow}$ from~$M_{\ell+1}$ to the matching~$M_{\ell}\in\cM_{\ell}$ obtained by replacing~$M_{\ell+1}[V(Y)]$ with~$Y$.
Let~$\delta_{\cM_{\ell+1}}^{e,\downarrow}$ denote the minimum degree in~$G_{e,\ell+1}^{\downarrow}$ among all~$M_{\ell+1}\in\cM_{\ell+1}$, and let~$\Delta_{\cM_{\ell}}^{e,\downarrow}$ denote the maximum degree in~$G_{e,\ell+1}^{\downarrow}$ among all~$M_{\ell}\in\cM_{\ell}$.
It is easy to see\COMMENT{since there are at most~$n^{m-\ell}$ choices of~$V(Y)$ containing the vertices of each of the~$\ell$ edges of any~$M_{\ell}\in\cM_{\ell}$ intersecting~$e$, and at most~$(mk)!$ ways to choose~$Y$ once~$V(Y)$ is fixed.}
that $\Delta_{\cM_{\ell}}^{e,\downarrow}\leq (mk)!n^{m-\ell}$.
Now fix~$M_{\ell+1}\in\cM_{\ell+1}$ and let $U(M_{\ell+1})\coloneqq\{e'\in M_{\ell+1}\colon e\cap e' \neq \emptyset\}$.
We apply Lemma~\ref{tidytool} again (with $\cP=M_{\ell+1}$, $\cP_{0}=U(M_{\ell+1})$, and with~$m-\ell-1$,~$k$,~$\ell+1$ playing the roles of~$m$,~$\ell$,~$t$, respectively) to deduce that $\delta_{\cM_{\ell+1}}^{e,\downarrow}\geq (mk)^{-m}n^{m-\ell-1}$, and thus $|\cM_{\ell}|/|\cM_{\ell+1}|\geq1/((mk)!(mk)^{m}n)\geq1/(Bn)$.

Finally, note that
\[
\prob{e\in M} = \frac{|\cM_{1}|}{|\cM_{1}|+\dots+|\cM_{k}|} \leq \frac{|\cM_{1}|}{|\cM_{k}|}= \frac{|\cM_{1}|}{|\cM_{2}|}\cdot\frac{|\cM_{2}|}{|\cM_{3}|}\dots\frac{|\cM_{k-1}|}{|\cM_{k}|} \leq\frac{B^{k-1}}{n^{k-1}}\leq \frac{C}{n^{k-1}},
\]
and similarly $\prob{e\in M}\geq |\cM_{1}|/(k|\cM_{k}|)\geq1/(kB^{k-1}n^{k-1})\geq 1/(Cn^{k-1})$.
\end{proof}
Finally, we use Lemma~\ref{perfect} to show that an $(n,k,\gamma)$-graph admits a normal perfect fractional matching.
\begin{lemma}\label{normality}
Let $1/n \ll 1/C \ll \gamma, 1/k$, where~$k \geq 2$, and let~$G$ be an $(n,k,\gamma)$-graph.
Then there exists a $C$-normal perfect fractional matching of~$G$.
\end{lemma}
\begin{proof}
Let~$i$ be the unique integer in $\{0,1,\dots,k-1\}$ satisfying $n \equiv i\bmod k$.
For each $\unord{S}\in\unordsubs{V}{i}$, let $G_{\unord{S}}\coloneqq G-\unord{S}$.
We define an edge weighting $\weighting{x}_{\unord{S}}\colon E(G_{\unord{S}})\rightarrow\bR^{+}$ by setting $\weighting{x}_{\unord{S}}(e)\coloneqq\prob{e\in M_{\unord{S}}}$ for each~$e\in E(G_{\unord{S}})$, where~$M_{\unord{S}}$ is a uniformly random perfect matching in~$G_{\unord{S}}$.
We define an edge weighting $\weighting{x}\colon E(G)\rightarrow\bR^{+}$ by setting
\[
\weightel{x}{e}\coloneqq\binom{n-1}{i}^{-1}\sum_{\unord{S}\in\unordsubs{V}{i}}\weighting{x}_{\unord{S}}(e),
\]
for each~$e\in E(G)$, where we set~$\weighting{x}_{\unord{S}}(e)$ to be~$0$ for each~$\unord{S}$ such that~$e\notin E(G_{\unord{S}})$.
Then, by Lemma~\ref{perfect},~$\weighting{x}$ is the desired $C$-normal perfect fractional matching of~$G$.
\COMMENT{It is clear that~$G_{\unord{S}}$ is~$\gamma/2$-Dirac for each $\unord{S}\in\unordsubs{V}{i}$.
Fix $\unord{S}\in\unordsubs{V}{i}$.
Applying Lemma \ref{perfect} to~$G_{\unord{S}}$ (and letting~$\gamma/2$ and~$C/2$ play the roles of~$\gamma$ and~$C$ respectively) we obtain $2/(C(n-i)^{k-1})\leq\weighting{x}_{\unord{S}}(e)\leq C/(2(n-i)^{k-1})$ for each~$e\in E(G_{\unord{S}})$.
Further, for any~$v\in V(G_{\unord{S}})$ we obtain that the sum of the weights on all edges of~$G_{\unord{S}}$ containing~$v$ is~$1$,
since any perfect matching of~$G_{\unord{S}}$ contains precisely one edge which contains~$v$.
Thus,~$\weighting{x}_{\unord{S}}$ is a perfect fractional matching of~$G_{\unord{S}}$.
It suffices to show that~$\weighting{x}$ is a perfect fractional matching of~$G$, and that (\ref{eq:doublebound}) holds.
We first argue that~$\weighting{x}$ is a perfect fractional matching of~$G$.
Fix~$v\in V$.
Note that there are precisely~$\binom{n-1}{i}$ choices of $\unord{S}\in\unordsubs{V}{i}$ such that~$v\in V(G_{\unord{S}})$, and for each such~$\unord{S}$ we have that~$\weighting{x}_{\unord{S}}$ is a perfect fractional matching of~$G_{\unord{S}}$.
Thus
we see from the definition of~$\weighting{x}$ that~$\weighting{x}$ is a perfect fractional matching of~$G$.
Finally we show that~$\weighting{x}$ satisfies (\ref{eq:doublebound}).
Fix~$e \in E(G)$.
Note that
\[
\weightel{x}{e} \leq \binom{n-1}{i}^{-1}\sum_{\substack{\unord{S}\in\unordsubs{V}{i}\\e\in E(G_{\unord{S}})}}\frac{C}{2(n-i)^{k-1}}\leq \frac{C}{2(n-i)^{k-1}} \leq \frac{C}{n^{k-1}}.
\]
The lower bound proceeds similarly
.}
\end{proof}
\section{Counting short paths}\label{countingshort}
The aim of this section is to prove the following lemma, which guarantees many short tight paths in a $\gamma$-Dirac $k$-graph~$G$, such that the $\gamma$-Dirac property of the graph~$G'$ obtained from deleting any such path is still essentially intact.
\begin{lemma}[Iteration Lemma]\label{iteration}
Let $1/n \ll c \ll \gamma,1/k$ where $k \geq 2$, let~$G$ be an $(n,k,\gamma)$-graph, and let $\ordered{S}\in\ordsubs{V}{k-1}$.
There exists a set~$\cP$ of $\sqrt{n}$-paths in~$G$ such that:
\begin{enumerate}[label=\upshape(\roman*)]
    \item $|\cP|\geq (cn)^{\sqrt{n}}$;
    \item $\ordered{S}$ is an end of each $P\in\cP$;
    \item if $P\in\cP$ and $\ordered{T}$ is the non-$\ordered{S}$ end of $P$, then $G'\coloneqq G-(V(P)\setminus \unord{T})$ satisfies $\delta(G')\geq(1/2 +\gamma-n^{-2/3})(n-\sqrt{n})$.
\end{enumerate}
\end{lemma}
We now provide some important definitions and sketch the proof of Lemma \ref{iteration}.
We then collect together a number of technical lemmas, and finally use these results to prove Lemma \ref{iteration}.
\subsection{Random walk notation}
We first define some random walks which will be of central importance to the proof of Lemma \ref{iteration}.
Let~$k\geq 2$ be an integer, let~$G$ be a $k$-graph, and let $\weighting{x}\colon E(G)\rightarrow\bR^{+}$ be a positive edge weighting function.
Each of our random walks $\cZ=(Z_{-(k-2)}, Z_{-(k-3)},\dots)$ on~$V$ will begin with an ordered $(k-1)$-tuple~$(Z_{-(k-2)},\dots, Z_{0})$ chosen according to some probability distribution $\mu\colon\ordsubs{V}{k-1}\rightarrow\bR^{+}$.
We say that~$\mu$ is the \textit{initial distribution} of~$\cZ$.
Random vertices will then be added one-by-one to each~$\cZ$ according to the transition probabilities of~$\cZ$.
Suppose we are given a random walk $\cZ=(Z_{-(k-2)},\dots, Z_{j-1})$ on~$V$, up to time~$j-1$.
Then we say that \textit{semiviable} vertices for step~$j$ are those vertices~$v\in V$ satisfying $\{Z_{j-(k-1)},\dots,Z_{j-1}\}\cup\{v\}\in E(G)$.
We say that \textit{viable} vertices for step~$j$ are those vertices~$v\in V$ which are semiviable and satisfy $v\notin\{Z_{-(k-2)},\dots,Z_{j-1}\}$.
Let~$Q_{j}$ and~$R_{j}$ denote the sets of semiviable and viable vertices for step~$j$, given the random walk up to time~$j-1$, respectively.

We say that a random walk $\cX=(X_{-(k-2)},X_{-(k-3)},\dots)$ on the vertices of~$G$, with any initial distribution~$\mu$, is a \textit{self-avoiding}~$\weighting{x}$-\textit{walk} to mean that the transition probabilities of~$\cX$ for~$j\geq1$ are defined for all $v\in V$ by
\[
\prob{X_{j}=v \mid X_{-(k-2)},\dots,X_{j-1}} \coloneqq \frac{\weightel{x}{\{X_{j-(k-1)},\dots,X_{j-1}\}\cup\{v\}}\mathds{1}_{R_{j}}(v)}{\sum_{w\in R_{j}}\weightel{x}{\{X_{j-(k-1)},\dots,X_{j-1}\}\cup\{w\}}},
\]
whenever~$R_{j}$ is non-empty, otherwise we terminate the walk.
Here, and throughout, we define $\weightel{x}{S}\coloneqq 0$ for any $S\notin E(G)$.
Note that $\cX=(X_{-(k-2)},X_{-(k-3)},\dots)$ is equivalent to the random walk $\cX^{(k-1)}=(\ordscript{X}{0},\ordscript{X}{1},\dots)$,
where each~$\ordscript{X}{i}$ is the ordered $(k-1)$-tuple $\ordscript{X}{i}= (X_{i-(k-2)},\dots,X_{i})$.
We thus refer to both~$\cX$ and~$\cX^{(k-1)}$ as the self-avoiding $\weighting{x}$-walk on~$G$ with initial distribution~$\mu$, since they are reformulations of each other.

Suppose now that~$G$ has no isolated $(k-1)$-tuples (this will always be true for us). We say that a random walk $\cY=(Y_{-(k-2)},Y_{-(k-3)},\dots)$ on the vertices of~$G$ (or $\cY^{(k-1)}=(\ordscript{Y}{0}, \ordscript{Y}{1},\dots)$ on~$\ordsubs{V}{k-1}$, where $\ordscript{Y}{i}\coloneqq(Y_{i-(k-2)},\dots, Y_{i})$), with any initial distribution~$\mu$, is a \textit{simple}~$\weighting{x}$-\textit{walk} to mean that the transition probabilities of~$\cY$ for~$j\geq1$ are defined by
\begin{equation} \label{eq:transprobsimple}
\prob{Y_{j}=v \mid Y_{-(k-2)},\dots,Y_{j-1}} \coloneqq \frac{\weightel{x}{\{Y_{j-(k-1)},\dots,Y_{j-1}\}\cup\{v\}}\mathds{1}_{Q_{j}}(v)}{\sum_{w\in Q_{j}}\weightel{x}{\{Y_{j-(k-1)},\dots,Y_{j-1}\}\cup\{w\}}},
\end{equation}
for all $v\in V$.
Note that~$\cY^{(k-1)}$ is a Markov chain on~$\ordsubs{V}{k-1}$ because the transition probabilities at any time depend only on the current state.
When the stationary distribution of~$\cY^{(k-1)}$ exists and is unique, we denote it by~$\pi$,
and we say that a simple $\weighting{x}$-walk $\cW^{(k-1)}=(\ordscript{W}{0},\ordscript{W}{1},\dots)$ on the ordered $(k-1)$-tuples of~$V$ is the \textit{stationary}~$\weighting{x}$-\textit{walk} on~$G$ if the initial distribution is~$\pi$.
Again, we have that $\cW^{(k-1)}=(\ordscript{W}{0},\ordscript{W}{1},\dots)$ is equivalent to the walk $\cW=(W_{-(k-2)},W_{-(k-3)},\dots)$ on the vertices of~$G$,
where each~$\ordscript{W}{i}$ is the ordered $(k-1)$-tuple $\ordscript{W}{i}= (W_{i-(k-2)},\dots,W_{i})$.
We also call~$\cW$ the stationary $\weighting{x}$-walk, and use the vertex (or tuple) version whenever it is more convenient.
Finally, whenever the initial distribution~$\mu$ of~$\cX$ (or~$\cY$) satisfies $\mu(\ordered{S})=1$ for some $\ordered{S}\in\ordsubs{V}{k-1}$,
we say that~$\cX$ (or~$\cY$) has \textit{starting tuple}~$\ordered{S}$.
\subsection{Further notation and sketch of the proof of Lemma \ref{iteration}}\label{propersketch}
We now describe our approach to proving Lemma \ref{iteration}.
Introduce a new constant~$C$ satisfying $1/n\ll c\ll1/C\ll\gamma,1/k$, and
let~$G$ be an $(n,k,\gamma)$-graph.
By Lemma \ref{normality}, there exists a $C$-normal perfect fractional matching~$\weighting{x}$ of~$G$.
We fix such a $C$-normal~$\weighting{x}$ throughout this proof sketch.
We will analyse a self-avoiding $\weighting{x}$-walk~$\cX$ on~$G$ with starting tuple~$\ordered{S}\in\ordsubs{V}{k-1}$.
We stop the walk after time~$\kappa\coloneqq \sqrt{n}$, so that we may write $\cX=(X_{-(k-2)},\dots,X_{\kappa})$.
Note that each outcome of~$\cX$ will correspond to a tight~$\kappa$-path in~$G$, with~$\ordered{S}$ as one end.

We define $V_{j}\coloneqq V\setminus\{X_{-(k-2)},\dots, X_{j-(k-1)}\}$ to be the set of all vertices of~$G$ except for all vertices~$\cX^{(k-1)}$ has visited strictly before~$\ordscript{X}{j}$.
We say that~$V_{j}$ is the \textit{residual vertex set} of~$G$ at time~$j$.
We also define $G_{j}\coloneqq G[V_{j}]$ and say that~$G_{j}$ is the \textit{residual graph} at time~$j$.
We also write~$\cX(j)$ to denote the walk~$\cX$ up to time~$j$, specifically $\cX(j)\coloneqq(X_{-(k-2)},\dots,X_{j})$.

We will show that it is likely that the $\gamma$-Dirac property of the residual graph~$G_{\kappa}$ is still essentially intact, by showing that it is likely that the vertices that~$\cX$ visits look roughly like a uniformly random subset of~$V$ (see Lemma~\ref{uniformity}).
For this, we will use the following `tracking functions' to monitor the progress of~$\cX$, with respect to how the codegree of each $(k-1)$-tuple in the residual graph deteriorates over time.

For each unordered $(k-1)$-tuple $\unord{S}\in\unordsubs{V}{k-1}$, we define a function $g_{\unord{S}}\colon V\rightarrow\bR^{+}$ by setting $g_{\unord{S}}(v)\coloneqq\mathds{1}_{N_{G}(\unord{S})}(v)$ for each~$v\in V$, so that, in particular, if $\unord{S}\subseteq V_{j}$ then~$g_{\unord{S}}(V_{j})=d_{G_{j}}(\unord{S})$.
We call the set $\cF\coloneqq\{g_{\unord{S}}\colon \unord{S}\in\unordsubs{V}{k-1}\}$ the set of \textit{tracking functions} of~$G$.
We say that $\cX=(X_{-(k-2)},\dots,X_{\kappa})$ is \textit{good} if 
\begin{equation}\label{eq:gooddef}
\sum_{i=-(k-2)}^{\kappa}g_{\unord{S}}(X_{i})=\frac{\kappa}{n}g_{\unord{S}}(V)\pm n^{3/10}\hspace{7mm}\text{for}\hspace{1mm}\text{all}\hspace{2mm}g_{\unord{S}}\in\cF.
\end{equation}
Thus, to say that~$\cX$ is good is to say that the set of~$\kappa+k-1$ vertices that~$\cX$ visits look roughly like a uniformly random subset of~$V$,
with respect to the codegrees of all $(k-1)$-tuples.
In particular, for all $(k-1)$-tuples $\unord{S}\in\unordsubs{V}{k-1}$, the proportion of vertices of~$N_{G}(\unord{S})$ visited by~$\cX$ is approximately~$\kappa/n$,
and this is the property that will allow us to deduce condition (iii) of Lemma~\ref{iteration}.

Let $(\ordscript{X}{a},\dots,\ordscript{X}{b})$ be a suitable interval of~$\cX^{(k-1)}$,
and let $\cY^{(k-1)}=(\ordscript{Y}{0}=\ordscript{X}{a},\ordscript{Y}{1},\dots)$ be the simple $\weighting{x}$-walk on~$G_{a}$ with starting tuple~$\ordscript{X}{a}$.
To show that the walk~$\cX$ is likely to be good, the following will be the main steps:
\begin{enumerate}[label=\upshape(\roman*)]
\item Firstly we show that the behaviour of $(\ordscript{X}{a},\dots)$ follows closely that of~$\cY^{(k-1)}$ by exhibiting a coupling of the two walks such that the probability of~$\cX^{(k-1)}$ and~$\cY^{(k-1)}$ being different is acceptably small, provided~$b-a$ is small.
\item Next we see that~$\cY^{(k-1)}$ \textit{mixes} (converges to its stationary distribution~$\pi$) rapidly.
\item We also show that the stationary $\weighting{x}$-walk $\cW=(W_{0},W_{1},\dots)$ satisfies $\prob{W_{i}=v}\approx1/|V_{a}|$ for each~$v\in V_{a}$.
\end{enumerate}
Putting the above together, we see that even for small~$q$, the distribution of each~$X_{i}$, given the walk to time~$i-q$, is typically close to the uniform distribution on~$V_{i-q}$,
and thus for each tracking function~$g\in\cF$, we have $\expn{g(X_{i})}\approx\frac{1}{n-(i-q)}g(V_{i-q})\approx \frac{1}{n}g(V)$.
Lastly, then:
\begin{enumerate}[label=\upshape(\roman*)]
\item[(iv)] We show that the actual values of the quantities $\sum_{a\leq i\leq b}g(X_{i})$ are very likely to be close to their expectations, by using Lemma \ref{azuma}.
\end{enumerate}
This completes the sketch of the proof that~$\cX$ is likely to be good.
It only remains to count the number of good walks (outcomes of)~$\cX$.
The count will be obtained by simply dividing a lower bound for the probability that~$\cX$ is good, by an upper bound for the probability of obtaining any specific outcome of~$\cX$.
This completes the sketch of the proof of Lemma~\ref{iteration}.
\subsection{Random walk analysis}
In this subsection we collect some of the tools that we will use to prove Lemma~\ref{iteration}.
We firstly define some convenient terminology for edge weightings.
Let $\weighting{x}\colon E(G)\rightarrow\bR^{+}$ be a positive edge weighting of a $k$-graph~$G$.
We say that~$\weighting{x}$ is $a_{1}$-\textit{lower-balanced} if for all non-isolated $\unord{S}\in\unordsubs{V}{k-1}$ and all~$v\in N_{G}(\unord{S})$ we have\COMMENT{We need~$\unord{S}$ non-isolated so that the denominator is not zero, we need that $v\in N_{G}(\unord{S})$ as the ratio is trivially zero for all other vertices~$v$. We can formulate the sum in the denominator to be over $v'\in V$ rather than $v'\in N_{G}(\unord{S})$ because we stated that $\weightel{x}{S}\coloneqq 0$ for any $S\notin E(G)$ throughout the paper immediately after the definition of the transition probabilities of~$\cX$ on page 10.} that $\weightel{x}{\unord{S}\cup\{v\}}/\left(\sum_{v'\in V}\weightel{x}{\unord{S}\cup\{v'\}}\right)\geq a_{1}$.
That is, all possible $\weighting{x}$-walk transition probabilities are bounded below by~$a_{1}$.
Similarly, we say that~$\weighting{x}$ is $a_{2}$-\textit{upper-balanced} if for all non-isolated $\unord{S}\in\unordsubs{V}{k-1}$ and all~$v\in N_{G}(\unord{S})$ we have $\weightel{x}{\unord{S}\cup\{v\}}/\left(\sum_{v'\in V}\weightel{x}{\unord{S}\cup\{v'\}}\right)\leq a_{2}$.
We say that~$\weighting{x}$ is $(a_{1},a_{2})$-\textit{balanced} if~$\weighting{x}$ is $a_{1}$-lower-balanced and $a_{2}$-upper-balanced.
We now give a simple result which shows that we may couple the self-avoiding $\weighting{x}$-walk and the simple $\weighting{x}$-walk to behave very similarly over small distances.
\begin{lemma}\label{walkalike}
Let $k\geq 2$, let~$G$ be a $k$-graph, let $\ordered{M}\in\ordsubs{V}{k-1}$, and let $\weighting{x}\colon E(G)\rightarrow\bR^{+}$ be a positive $r$-upper-balanced edge weighting.
Let~$\cX=(X_{-(k-2)},\dots)$ and~$\cY=(Y_{-(k-2)},\dots)$ be, respectively, the self-avoiding $\weighting{x}$-walk and the simple $\weighting{x}$-walk on~$G$, each with starting tuple~$\ordered{M}$.
Then for any positive integer~$q\leq \delta(G)$, we have
\[
d_{TV}(X_{q}, Y_{q})\leq q^{2}r.
\]
\end{lemma}
\begin{proof}
If\COMMENT{The hypothesis $\delta(G)\geq q$ ensures that~$\cX$ cannot have terminated (and that there are no isolated $(k-1)$-tuples, as promised when we defined the transition probabilities of~$\cY$). Indeed, even the tuple~$\{v_{q-(k-1)},\dots,v_{q-1}\}$ must have at least one neighbour not previously visited by the walks, so that there is a viable candidate for~$v_{q}$.}~$1\leq i\leq \delta(G)$, and~$\cX$ and~$\cY$ agree up to time~$i-1$, say $X_{j}=Y_{j}=v_{j}\in V$ for each~$j\in\{-k+2, \dots, i-1\}$, then we couple at the next step so that~$X_{i}$ coincides with~$Y_{i}$ whenever the choice of~$Y_{i}$ is a viable choice for~$X_{i}$,
which is to say that~$Y_{i}$ is not a vertex already seen.
So with this coupling, for any positive integer $i\leq\delta(G)$ we have
\[
\prob{X_{i}\neq Y_{i}\mid X_{j}= Y_{j} \hspace{2mm} \text{for all} \hspace{2mm} j\in\{-k+2, \dots, i-1\}}=\prob{Y_{i}\in\{v_{-(k-2)},\dots,v_{i-k}\}}\leq ir.
\]
Thus for any positive integer~$q\leq\delta(G)$ we obtain
\[
\prob{X_{q}\neq Y_{q}} \leq \sum_{i=1}^{q}\prob{X_{i}\neq Y_{i}\mid X_{j}= Y_{j} \hspace{2mm} \text{for all} \hspace{2mm} j\in\{-k+2, \dots, i-1\}} \leq q^{2}r.
\]
The desired result now follows from~(\ref{eq:variation2}).
\end{proof}
We now aim to show that~$\cY^{(k-1)}$ mixes rapidly.
The key part of the proof will be the following argument that~$\cY^{(k-1)}$ has many different choices for how to arrive at a specified target ordered tuple~$\ordered{T}\in\ordsubs{V}{k-1}$ in a fixed number of steps.
\begin{lemma}\label{manywalks}
Let $1/n\ll\zeta\ll1/\ell\ll\gamma,1/k$, where~$k\geq2$, and let~$G$ be an $(n,k,\gamma)$-graph.
For any~$\ordered{S},\ordered{T}\in\ordsubs{V}{k-1}$, there are at least $\zeta n^{\ell-(k-1)}$~$\ell$-walks from~$\ordered{S}$ to~$\ordered{T}$ in~$G$.
Further, if~$\unord{S}\cap\unord{T}=\emptyset$, then there are at least $\zeta n^{\ell-(k-1)}$~$\ell$-paths in~$G$ which connect~$\ordered{S}$ and~$\reversed{T}$.
\end{lemma}
\begin{proof}
Let $\ordered{S},\ordered{T}\in\ordsubs{V}{k-1}$,
set $a\coloneqq\ell-(k-1)$,
let~$\cQ\coloneqq\binom{V\setminus (\unord{S}\cup\unord{T})}{a}$,
let $G_{Q}\coloneqq G[\unord{S}\cup Q\cup\unord{T}]$ for each~$Q\in\cQ$,
and define $\cQ'\coloneqq\{Q\in\cQ\colon G_{Q} \hspace{1mm} \text{is} \hspace{1mm} (\gamma/2)\text{-Dirac}\}$.
We apply Proposition~\ref{tidytool} (with $\cP=\{\{v\}\colon v\in V\}$, $\cP_{0}=\{\{v\}\colon v\in\unord{S}\cup\unord{T}\}$, and with~$a$,~$1$,~$|\unord{S}\cup\unord{T}|$ playing the roles of~$m$,~$\ell$,~$t$, respectively) to deduce that $|\cQ'|\geq |\cQ|/2$.
Fix~$Q\in\cQ'$ and write $\ordered{T}=(t_{1},\dots,t_{k-1})$.
Since $G[Q\cup \unord{T}]$ is $\gamma/4$-Dirac, for each $i\in [k-1]$ in turn we may find a vertex $v_{i}\in Q\setminus\{v_{1},\dots,v_{i-1}\}$ such that $\{v_{1},\dots,v_{i}\}\cup\{t_{1},\dots,t_{k-i}\}$ is an edge.
Write $\ordered{T'}\coloneqq(v_{k-1},\dots,v_{1})$.
Thus, we obtain a $(k-1)$-path $P_{Q}^{2}$ in~$G_{Q}$ which connects~$\ordered{T'}$ and~$\reversed{T}$.
Since $G[\unord{S}\cup Q]$ is $\gamma/4$-Dirac, we may apply Lemma~\ref{hampath} (with~$\gamma/4$ playing the role of~$\gamma$) to find an $a$-path~$P_{Q}^{1}$ in $G[\unord{S}\cup Q]$ which connects~$\ordered{S}$ and~$\reversed{T'}$,
and the obvious concatenation of~$P_{Q}^{1}$ and~$P_{Q}^{2}$ is an $\ell$-walk~$W_{Q}$ from~$\ordered{S}$ to~$\ordered{T}$ in~$G_{Q}$.
It is clear that these $\ell$-walks~$W_{Q}$ are distinct for different choices of~$Q\in\cQ'$.
It thus suffices to observe that $|\cQ'| \geq \frac{1}{2}|\cQ| \geq \zeta n^{\ell-(k-1)}$.
If $\unord{S}\cap\unord{T}=\emptyset$, then the walks~$W_{Q}$ do not revisit vertices and thus correspond to $\ell$-paths in~$G$ which connect~$\ordered{S}$ and~$\reversed{T}$.
\end{proof}
Lemma~\ref{manywalks} shows that in an $(n,k,\gamma)$-graph~$G$,
the Markov chain~$\cY^{(k-1)}$ is irreducible, and thus there is a unique stationary distribution~$\pi$ of~$\cY^{(k-1)}$.
We will use this fact without stating it from now on.
We note that it also follows from Lemma~\ref{manywalks} that~$\cY^{(k-1)}$ is aperiodic\COMMENT{Let $\ordered{S}\in\ordsubs{V}{k-1}$ be a state of~$\cY^{(k-1)}$. We apply Lemma~\ref{manywalks} with $\ell=m$, say, to see that $\prob{\ordscript{Y}{m}=\ordered{S}\mid\ordscript{Y}{0}=\ordered{S}}>0$. We then apply Lemma~\ref{manywalks} with $\ell=m+1$ (assuming~$n$ is large enough for both of these applications), to see that $\prob{\ordscript{Y}{m+1}=\ordered{S}\mid\ordscript{Y}{0}=\ordered{S}}>0$. Thus the period of~$\ordered{S}$ is~$1$. As~$\ordered{S}$ was arbitrary, we conclude that~$\cY^{(k-1)}$ is aperiodic.}, which implies that the distribution of~$\ordscript{Y}{t}$ converges to~$\pi$ as $t\rightarrow \infty$. However, we need something stronger, namely that this convergence occurs quickly.
This is achieved by the following lemma, which shows that~$\cY^{(k-1)}$ mixes rapidly.
\begin{lemma}[Mixing Lemma]\label{mixing}
Let $1/n \ll 1/\lambda \ll \gamma,\tau,1/k$, where $k \geq 2$, let~$G$ be an $(n,k,\gamma)$-graph, let $\mu\colon\ordsubs{V}{k-1}\rightarrow\bR$ be a probability distribution, and let $\weighting{x}\colon E(G)\rightarrow\bR^{+}$ be a positive $(\tau/n)$-lower-balanced edge weighting.
Let~$\cY=(Y_{-(k-2)},\dots)$ be the simple $\weighting{x}$-walk on~$G$ with initial distribution~$\mu$, and let~$\cW=(W_{-(k-2)},\dots)$ be the stationary $\weighting{x}$-walk on~$G$.
For any~$\eta>0$, if $q\geq\lambda\ln(1/\eta)$, then $d_{TV}(Y_{q}, W_{q}) <\eta$.
\end{lemma}
\begin{proof}
Choose new constants~$\zeta$ and~$\ell$ satisfying $1/n \ll 1/\lambda \ll \zeta\ll1/\ell\ll\gamma,\tau,1/k$.
We proceed by using Lemma \ref{manywalks} to show that for two simple $\weighting{x}$-walks~$\cZ^{(k-1)}$ and~$\cZ'^{(k-1)}$ on~$G$ given any initial distributions, we can find a coupling such that~$\cZ^{(k-1)}$ and~$\cZ'^{(k-1)}$ are relatively likely to meet after~$\ell$ steps.
Using this, we then exhibit a coupling of~$\cY^{(k-1)}$ and~$\cW^{(k-1)}$ that will allow us to use (\ref{eq:variation2}) to upper bound~$d_{TV}(Y_{q}, W_{q})$.

Let~$\cZ^{(k-1)}$ be a simple $\weighting{x}$-walk on~$G$ with any initial distribution~$\phi$, and fix any~$\ordered{T}\in\ordsubs{V}{k-1}$.
By Lemma \ref{manywalks},~$G$ has at least~$\zeta n^{\ell-(k-1)}$ $\ell$-walks from~$\ordered{S}$ to~$\ordered{T}$, for each~$\ordered{S}\in\ordsubs{V}{k-1}$, and thus:
\begin{eqnarray*}
\prob{\ordscript{Z}{\ell}=\ordered{T}} & = & \sum_{\ordered{S}\in\ordsubs{V}{k-1}}\phi(\ordered{S})\prob{\ordscript{Z}{\ell}=\ordered{T}\mid\ordscript{Z}{0}=\ordered{S}}\geq\sum_{\ordered{S}\in\ordsubs{V}{k-1}}\phi(\ordered{S})\zeta n^{\ell-(k-1)}(\tau/n)^{\ell} \\ & = & \zeta\tau^{\ell}/n^{k-1}.
\end{eqnarray*}
Thus we may construct a coupling of any pair of simple $\weighting{x}$-walks~$\cZ^{(k-1)}$ and~$\cZ'^{(k-1)}$ such that $\prob{\ordscript{Z}{\ell}=\ordscript{Z}{\ell}'=\ordered{T}}\geq \zeta\tau^{\ell}/n^{k-1}$ for all $\ordered{T}\in\ordsubs{V}{k-1}$.\COMMENT{To construct such a coupling~$h(\ordscript{Z}{\ell}, \ordscript{Z}{\ell}')$ of the random variables~$\ordscript{Z}{\ell}$ and~$\ordscript{Z}{\ell}'$, we can set $h(\ordered{T},\ordered{T})=\zeta\tau^{\ell}/n^{k-1}$ for all $\ordered{T}\in\ordsubs{V}{k-1}$ (constructing the rest of this coupling arbitrarily). To see (informally) how this corresponds to a formal coupling of the random walks~$\cZ^{(k-1)}$ and~$\cZ'^{(k-1)}$, one can partition the unit interval into~$|\ordsubs{V}{k-1}|$ intervals of size~$\zeta\tau^{\ell}/n^{k-1}$ and the remaining interval. The idea is that we associate each of the first $|\ordsubs{V}{k-1}|$ intervals with some $\ordered{T}\in\ordsubs{V}{k-1}$, and partition each of these intervals into~$\zeta n^{\ell-(k-1)}$ sub-intervals of size~$(\tau/n)^{\ell}$. Fix $\ordered{T}\in\ordsubs{V}{k-1}$ and order these sub-intervals $I_{\ordered{T}}(1),\dots, I_{\ordered{T}}(\zeta n^{\ell-(k-1)})$. We arbitrarily assign (injectively) outcomes $\cZ_{\ordered{T}}(i)$ of~$\cZ^{(k-1)}(\ell)$ satisfying $\ordscript{Z}{\ell}=\ordered{T}$ to the intervals~$I_{\ordered{T}}(i)$, and we do the same for $\cZ'^{(k-1)}(\ell)$. We complete the mapping arbitrarily (to a valid coupling) in the interval not corresponding to any~$\ordered{T}$. Then we consider a uniform random variable~$X$ on the interval~$[0,1]$ and set $\cZ^{(k-1)}(\ell)=\cZ_{\ordered{T}}(i)$ and $\cZ'^{(k-1)}(\ell)=\cZ'_{\ordered{T}}(i)$ if $X\in I_{\ordered{T}}(i)$.}
Under this coupling, we have
\begin{equation}\label{eq:zcoupling}
\prob{\ordscript{Z}{\ell}=\ordscript{Z'}{\ell}}\geq\frac{\zeta\tau^{\ell}}{n^{k-1}}\left|\ordsubs{V}{k-1}\right|\geq \zeta\tau^{\ell}\left(\frac{n-(k-2)}{n}\right)^{k-1}\geq\zeta\tau^{\ell}/2.
\end{equation}
We now construct a coupling of~$\cY^{(k-1)}$ and~$\cW^{(k-1)}$ as follows.
We partition the time steps into consecutive intervals of length~$\ell$.
In the first interval, we couple~$\cY^{(k-1)}$ and~$\cW^{(k-1)}$ as in~(\ref{eq:zcoupling}), so that $\prob{\ordscript{Y}{\ell}=\ordscript{W}{\ell}}\geq\zeta\tau^{\ell}/2$.
If $\ordscript{Y}{\ell}=\ordscript{W}{\ell}$, then we couple~$\cY^{(k-1)}$ and~$\cW^{(k-1)}$ such that $\ordscript{Y}{t}=\ordscript{W}{t}$ for all $t\geq\ell$.
Otherwise we again couple~$\cY^{(k-1)}$ and~$\cW^{(k-1)}$ in the second time interval, as in~(\ref{eq:zcoupling}), so that $\prob{\ordscript{Y}{2\ell}=\ordscript{W}{2\ell}\mid\ordscript{Y}{\ell}\neq\ordscript{W}{\ell}}\geq\zeta\tau^{\ell}/2$.
One can easily check\COMMENT{One only needs to check that the transition probabilities for each walk are maintained, which is clear whether the walks meet or not.} that repeating this process yields a valid coupling of~$\cY^{(k-1)}$ and~$\cW^{(k-1)}$.
Note that, with this coupling and for any~$q$ which is sufficiently large compared to~$\ell$, we have
\[
\prob{\ordscript{Y}{q}\neq\ordscript{W}{q}} \leq\prod_{k=1}^{q/\ell}\prob{\ordscript{Y}{k\ell}\neq\ordscript{W}{k\ell}\bigg|\bigcap_{j\leq k-1}\{\ordscript{Y}{j\ell}\neq\ordscript{W}{j\ell}\}} \leq (1-\zeta\tau^{\ell}/2)^{q/\ell}.
\]
We use this coupling and apply~(\ref{eq:variation2}) to obtain\COMMENT{Note clearly $(Y_{q}\neq W_{q})\implies (\ordscript{Y}{q}\neq\ordscript{W}{q})$.}
\[
d_{TV}(Y_{q}, W_{q})\leq\prob{Y_{q}\neq W_{q}}\leq\prob{\ordscript{Y}{q}\neq\ordscript{W}{q}}\leq (1-\zeta\tau^{\ell}/2)^{q/\ell}\leq\exp(-\zeta\tau^{\ell}q/2\ell)<\eta,
\]
provided $q\geq\lambda\ln(1/\eta)$.
\end{proof}
It will be useful to have an explicit formula for the stationary distribution~$\pi$ of~$\cY^{(k-1)}$, and so we obtain this now.
(Observe that the simple $\weighting{x}$-walk~$\cY^{(k-1)}$ on~$\ordsubs{V}{k-1}$ is in general not a symmetric Markov chain.)
\begin{prop}\label{explicitstatdist}
Let $1/n \ll \gamma \ll 1/k$ where $k \geq 2$, and let~$G$ be an $(n,k,\gamma)$-graph.
Let $\weighting{x}\colon E(G)\rightarrow\bR^{+}$ be a positive edge weighting of~$G$, and for each~$\ordered{M}\in\ordsubs{V}{k-1}$, define
\begin{equation}\label{eq:explicit}
\pi(\ordered{M}) \coloneqq \frac{\sum_{v\in V}\weightel{x}{\unord{M}\cup\{v\}}}{\sum_{\ordered{B}\in\ordsubs{V}{k-1}}\sum_{v\in V}\weightel{x}{\unord{B}\cup\{v\}}}.
\end{equation}
Then~$\pi$ is the unique stationary distribution of the simple $\weighting{x}$-walk~$\cY^{(k-1)}$ on~$\ordsubs{V}{k-1}$.
\end{prop}
\begin{proof}
By standard results on the stationary distribution of a Markov chain (see~\cite[Proposition 1.20]{L17} for example), it suffices to prove that $\pi\colon\ordsubs{V}{k-1}\rightarrow\bR^{+}$ as defined in (\ref{eq:explicit}) is a probability distribution on~$\ordsubs{V}{k-1}$, and that
\begin{equation}\label{eq:firstly}
\sum_{\ordered{S}\in\ordsubs{V}{k-1}}\pi(\ordered{S})P(\ordered{S},\ordered{T}) = \pi(\ordered{T}) \hspace{7mm}\text{for}\ \text{all}\ \ordered{T}\in\ordsubs{V}{k-1},
\end{equation}
where~$P(\ordered{S},\ordered{T})$ denotes the (one-step) transition probability of~$\cY^{(k-1)}$ from~$\ordered{S}$ to~$\ordered{T}$.
It follows quickly\COMMENT{Clearly $\pi(\ordered{M})\geq 0$ for all $\ordered{M}\in\ordsubs{V}{k-1}$, and we have
\[
\sum_{\ordered{M}\in\ordsubs{V}{k-1}}\pi(\ordered{M}) = \sum_{\ordered{M}\in\ordsubs{V}{k-1}}\frac{\sum_{v\in V}\weightel{x}{\unord{M}\cup\{v\}}}{\sum_{\ordered{B}\in\ordsubs{V}{k-1}}\sum_{v\in V}\weightel{x}{\unord{B}\cup\{v\}}} = 1.
\]}
from (\ref{eq:explicit}) that~$\pi$ is a probability distribution on~$\ordsubs{V}{k-1}$, and~(\ref{eq:firstly}) follows\COMMENT{Fix $\ordered{T}\in\ordsubs{V}{k-1}$. 
Write $\ordered{T}=(t_{1},\dots, t_{k-1})$ and define~$\text{Pre}(\ordered{T})\coloneqq\{\ordered{M}=(m_{1},\dots, m_{k-1})\in\ordsubs{V}{k-1}\colon \unord{M}\cup\unord{T}\in E(G), t_{i}=m_{i+1}\,\,\, \forall \,\,\,1\leq i\leq k-2\}$,
so that~$\text{Pre}(T)$ is the set of possible predecessors of~$\ordered{T}$ in~$\cY^{(k-1)}$.
Similarly, define $\text{Succ}(\ordered{T})\coloneqq\{\ordered{M}=(m_{1},\dots, m_{k-1})\in\ordsubs{V}{k-1}\colon \unord{M}\cup\unord{T}\in E(G), m_{i}=t_{i+1}\,\,\, \forall \,\,\,1\leq i\leq k-2\}$.
Notice firstly that
\[
\sum_{\ordered{S}\in\ordsubs{V}{k-1}}\weightel{x}{\unord{S}\cup\unord{T}}\mathds{1}_{\text{Pre}(\ordered{T})}(\ordered{S})=\sum_{e\supset \unord{T}}\weightel{x}{e}=\sum_{\ordered{S}\in\ordsubs{V}{k-1}}\weightel{x}{\unord{S}\cup\unord{T}}\mathds{1}_{\text{Succ}(\ordered{T})}(\ordered{S}).
\]
Thus, applying~(\ref{eq:transprobsimple}) and~(\ref{eq:explicit}), we see that
\begin{eqnarray*}
\sum_{\ordered{S}\in\ordsubs{V}{k-1}}\pi(\ordered{S})P(\ordered{S},\ordered{T}) & = & \sum_{\ordered{S}\in\ordsubs{V}{k-1}}\frac{\sum_{v\in V}\weightel{x}{\unord{S}\cup\{v\}}}{\sum_{\ordered{B}\in\ordsubs{V}{k-1}}\sum_{v\in V}\weightel{x}{\unord{B}\cup\{v\}}}\frac{\weightel{x}{\unord{S}\cup \unord{T}}\mathds{1}_{\text{Pre}(\ordered{T})}(\ordered{S})}{\sum_{v\in V}\weightel{x}{\unord{S}\cup\{v\}}} \\ & = & \frac{1}{\sum_{\ordered{B}\in\ordsubs{V}{k-1}}\sum_{v\in V}\weightel{x}{\unord{B}\cup\{v\}}}\sum_{\ordered{S}\in\ordsubs{V}{k-1}}\weightel{x}{\unord{S} \cup \unord{T}}\mathds{1}_{\text{Pre}(\ordered{T})}(\ordered{S}) \\ & = & \frac{1}{\sum_{\ordered{B}\in\ordsubs{V}{k-1}}\sum_{v\in V}\weightel{x}{\unord{B}\cup\{v\}}}\sum_{\ordered{S}\in\ordsubs{V}{k-1}}\weightel{x}{\unord{S} \cup \unord{T}}\mathds{1}_{\text{Succ}(\ordered{T})}(\ordered{S}) \\ & = & \sum_{\ordered{S}\in\ordsubs{V}{k-1}}\pi(\ordered{T})P(\ordered{T},\ordered{S})=\pi(\ordered{T}).
\end{eqnarray*}} from applying~(\ref{eq:transprobsimple}) and~(\ref{eq:explicit}).
\end{proof}
Next we show that, provided the edge weighting $\weighting{x}\colon E(G)\rightarrow\bR^{+}$ is an `almost-perfect' fractional matching, the stationary $\weighting{x}$-walk on~$G$ is such that each vertex of~$G$ is roughly equally likely to be the current vertex at any time.
Let $\eps\in (0,1)$ and $k\geq2$.
For a $k$-graph~$G$, we say that a fractional matching $\weighting{x}\colon E(G)\rightarrow\bR^{+}$ is $\eps$-\textit{almost-perfect} if $\sum_{e\ni v}\weightel{x}{e}\geq 1-\eps$ for all~$v\in V$.
\begin{lemma}\label{welltravelled}
Let $1/n\ll \gamma, 1/k$, where $k\geq 2$, and let~$G$ be an $(n,k,\gamma)$-graph.
Suppose that $\weighting{x}\colon E(G)\rightarrow\bR^{+}$ is an $n^{-2/5}$-almost-perfect fractional matching of~$G$ and let $\cW=(W_{-(k-2)}, W_{-(k-3)}, \dots)$ be the stationary $\weighting{x}$-walk on~$G$.
Then for each~$i\in\bN$ and each $v\in V$ we have
\begin{equation}\label{eq:nearlyuniform}
\prob{W_{i}=v}=(1\pm n^{-1/3})\cdot\frac{1}{n}.
\end{equation}
\end{lemma}
\begin{proof}
We reformulate~$\cW$ as $\cW=(\ordscript{W}{0}, \ordscript{W}{1},\dots)$, where $\ordscript{W}{i}\coloneqq(W_{i-(k-2)},\dots, W_{i})$.
Let $v\in V$ and $i\in\bN$.
By the law of total probability,~(\ref{eq:transprobsimple}), and Proposition~\ref{explicitstatdist}, we obtain\COMMENT{In the third equality we exchange sums over ordered $(k-1)$-tuples for sums over unordered $(k-1)$-tuples, and a factor of $(k-1)!$ cancels from the numerator and denominator of the fraction.}:
\begin{eqnarray*}
\prob{W_{i}=v} & = & \sum_{\ordered{S}\in\ordsubs{V}{k-1}}\prob{\ordscript{W}{i-1}=\ordered{S}}\prob{W_{i}=v\mid\ordscript{W}{i-1}=\ordered{S}} \\
& = & \sum_{\ordered{S}\in\ordsubs{V}{k-1}}\frac{\sum_{v'\in V}\weightel{x}{\unord{S}\cup\{v'\}}}{\sum_{\ordered{B}\in\ordsubs{V}{k-1}}\sum_{v'\in V}\weightel{x}{\unord{B}\cup\{v'\}}}\cdot\frac{\weightel{x}{\unord{S}\cup\{v\}}}{\sum_{v'\in V}\weightel{x}{\unord{S}\cup\{v'\}}} \\
& = & \frac{\sum_{\unord{S}\in\unordsubs{V}{k-1}}\weightel{x}{\unord{S}\cup\{v\}}}{\sum_{\unord{B}\in\unordsubs{V}{k-1}}\sum_{v'\in V}\weightel{x}{\unord{B}\cup\{v'\}}}
= \frac{\sum_{\unord{S}\in\unordsubs{V}{k-1}}\weightel{x}{\unord{S}\cup\{v\}}}{\sum_{v'\in V}\sum_{\unord{B}\in\unordsubs{V}{k-1}}\weightel{x}{\unord{B}\cup\{v'\}}} \\
& = & \frac{\sum_{e\ni v}\weightel{x}{e}}{\sum_{v'\in V}\sum_{e\ni v'}\weightel{x}{e}}.
\end{eqnarray*}
Applying $1-n^{-2/5}\leq \sum_{e\ni v}\weightel{x}{e}\leq1$ for all $v\in V$, we obtain\COMMENT{Lower bound: $\sum_{e\ni v}\weightel{x}{e}/\sum_{v'\in V}\sum_{e\ni v'}\weightel{x}{e} \geq \frac{1-n^{-2/5}}{n}>(1-n^{-1/3})\cdot1/n$.\newline Upper bound: $\sum_{e\ni v}\weightel{x}{e}/\sum_{v'\in V}\sum_{e\ni v'}\weightel{x}{e} \leq \frac{1}{(1-n^{-2/5})n}$. Note that $\frac{1}{1-n^{-2/5}}=\frac{1-n^{-2/5}+n^{-2/5}}{1-n^{-2/5}}=1+\frac{n^{-2/5}}{1-n^{-2/5}}\leq1+2n^{-2/5}<1+n^{-1/3}$.}~(\ref{eq:nearlyuniform}).
\end{proof}
We now show  the crucial fact that, for any large set $U\subseteq V$, the probability that the self-avoiding $\weighting{x}$-walk~$\cX$ is in~$U$ after a small number of steps is roughly~$|U|/n$.
We will apply this fact to neighbourhoods of $(k-1)$-tuples in the proof that~$\cX$ is likely to be good.
This in turn will be used to show that~$\cX$ will, in expectation, behave roughly uniformly with respect to the codegrees of all $(k-1)$-tuples.
\begin{lemma}[Uniformity Lemma]\label{uniformity}
Let $1/n\ll 1/C\ll\gamma, 1/k$, where $k\geq2$, and let~$G$ be an $(n,k,\gamma)$-graph.
Let $\weighting{x}:E(G)\rightarrow\bR^{+}$ be a $C$-normal, $n^{-2/5}$-almost-perfect fractional matching, let $\ordered{S}\in\ordsubs{V}{k-1}$, and let $\cX=(X_{-(k-2)}, X_{-(k-3)},\dots)$ be the self-avoiding $\weighting{x}$-walk on~$G$ with starting tuple~$\ordered{S}$.
Then for any $q\in[(\ln n)^{2}, n^{1/5}]$ and any $U\subseteq V$ of size $|U|\geq n^{3/4}$, we have
\begin{equation}\label{eq:niceuniformity}
\prob{X_{q}\in U} =(1\pm n^{-3/10})\frac{|U|}{n}.
\end{equation}
\end{lemma}
\begin{proof}
We first argue that~$\weighting{x}$ is $(1/C^{2}n, 2C^{2}/n)$-balanced.
Indeed, since~$\weighting{x}$ is $C$-normal and~$G$ is $\gamma$-Dirac, we have
\[
\frac{\max_{e\in E(G)}\weightel{x}{e}}{\min_{\unord{S}\in\unordsubs{V}{k-1}}\sum_{v\in V}\weightel{x}{\unord{S}\cup\{v\}}}\leq\frac{C}{n^{k-1}}\cdot\frac{Cn^{k-1}}{(1/2+\gamma)n}\leq \frac{2C^{2}}{n}.
\]
The lower bound follows similarly\COMMENT{$\frac{\min_{e\in E(G)}\weightel{x}{e}}{\max_{\unord{S}\in\unordsubs{V}{k-1}}\sum_{v\in V}\weightel{x}{\unord{S}\cup\{v\}}}\geq \frac{1}{Cn^{k-1}}\cdot\frac{n^{k-1}}{Cn} =1/C^{2}n$.}.
Now let $U\subseteq V$ be of size $|U|\geq n^{3/4}$, and fix $q\in[(\ln n)^{2}, n^{1/5}]$.
Let $\cY=(Y_{-(k-2)}, Y_{-(k-3)},\dots)$ be the simple $\weighting{x}$-walk on~$G$ with starting tuple~$\ordered{S}$ and let $\cW=(W_{-(k-2)}, W_{-(k-3)}, \dots)$ be the stationary $\weighting{x}$-walk on~$G$.
We will show that~$X_{q}$ is distributed similarly to~$Y_{q}$ since~$q$ is not too large, and that~$Y_{q}$ is distributed similarly to~$W_{q}$ since~$q$ is large enough, and finally we will use Lemma~\ref{welltravelled} to show that~$\prob{W_{q}\in U}$ is roughly~$|U|/n$.

Setting $c_{1}\coloneqq |\prob{X_{q}\in U}-\prob{Y_{q}\in U}|$ and applying Lemma~\ref{walkalike} (with~$2C^{2}/n$ playing the role of~$r$), we obtain
\begin{eqnarray*}
c_{1} & = & \left|\sum_{v\in U}(\prob{X_{q}=v}-\prob{Y_{q}=v})\right| \leq\sum_{v\in V}|\prob{X_{q}=v}-\prob{Y_{q}=v}| \stackrel{(\ref{eq:variation2})}{=}2d_{TV}(X_{q}, Y_{q}) \\ & \leq & 4C^{2}q^{2}/n <\frac{1}{3}n^{-13/10}|U|.
\end{eqnarray*}
Consider the term $c_{2}\coloneqq |\prob{Y_{q}\in U}-\prob{W_{q}\in U}|$.
We apply Lemma~\ref{mixing} (with~$1/C^{2}$,~$1/n^{2}$ playing the roles of~$\tau$,~$\eta$ respectively, and using $q\geq (\ln n)^{2}=((\ln n)/2)\cdot\ln(1/\eta)$) to obtain
\[
c_{2}\leq\sum_{v\in V}|\prob{Y_{q}=v}-\prob{W_{q}=v}| \stackrel{(\ref{eq:variation2})}{=}2d_{TV}(Y_{q}, W_{q}) < 2n^{-2}<\frac{1}{3}n^{-13/10}|U|.
\]
We now apply Lemma~\ref{welltravelled} to the term $c_{3}\coloneqq |\prob{W_{q}\in U}-|U|/n|$ to obtain
\[
c_{3} =\left|\sum_{v\in U}\left(\prob{W_{q}=v}-\frac{1}{n}\right)\right| \leq \sum_{v\in U}\left|\prob{W_{q}=v}-\frac{1}{n}\right| \leq n^{-4/3}|U|.
\]
Finally, by the triangle inequality we have $|\prob{X_{q}\in U}-|U|/n|\leq c_{1}+c_{2}+c_{3}<n^{-13/10}|U|$, which is equivalent to~(\ref{eq:niceuniformity}).
\end{proof}
\subsection{The walk is likely to be good}\label{WALKGOOD}
The aim of this section is to prove the following lemma, which states that under our assumptions, a self-avoiding $\weighting{x}$-walk of length~$\sqrt{n}$ is good with high probability.
\begin{lemma}\label{goodwalk}
Let $1/n\ll 1/C\ll \gamma, 1/k$, where $k\geq 2$, and let~$G$ be an $(n,k,\gamma)$-graph.
Let $\weighting{x}\colon E(G)\rightarrow\bR^{+}$ be a $C$-normal perfect fractional matching, let $\ordered{S}\in\ordsubs{V}{k-1}$,
let $\kappa\coloneqq\sqrt{n}$, and
let $\cX=(X_{-(k-2)},\dots, X_{\kappa})$ be a self-avoiding $\weighting{x}$-walk on~$G$ with starting tuple~$\ordered{S}$.
Then~$\prob{\cX\hspace{1mm}\text{is}\hspace{1mm}\text{good}}\geq 1-1/n$.
\end{lemma}
To prove Lemma~\ref{goodwalk}, we will need some results on the behaviour of the residual graphs~$G_{j}$ as the walk~$\cX$ progresses, so that we can apply Lemma~\ref{uniformity} to each~$G_{j}$.
We define~$\weighting{x}|_{G_{j}}$ to be the restriction of~$\weighting{x}$ to~$E(G_{j})$, so that $\weighting{x}|_{G_{j}}\colon E(G_{j})\rightarrow\bR^{+}$ is a (not necessarily perfect) fractional matching of~$G_{j}$.
\begin{prop}\label{deterministic}
Suppose the assumptions of Lemma~\ref{goodwalk} hold.
Let~$\cF$ be the set of tracking functions of~$G$.
Then for any~$g\in\cF$ and any $j\in\{0,\dots,\kappa\}$, the following conditions hold deterministically:
\begin{enumerate}[label=\upshape(\roman*)]
    \item $g(V_{j})=\left(1\pm n^{-1/4}\right)\frac{n-j}{n}g(V)$;
    \item $G_{j}$ is $\gamma/2$-Dirac;
    \item $\weighting{x}|_{G_{j}}$ is $2C$-normal;
    \item $\weighting{x}|_{G_{j}}$ is $n^{-2/5}$-almost-perfect.
\end{enumerate}
\end{prop}
\begin{proof}
It suffices to prove that conditions (i)--(iv) hold for any $j\in\{0,\dots,\kappa\}$ and any outcome $x(j)=(x_{-(k-2)},\dots, x_{j})$ of~$\cX(j)$.
Throughout the proof, we let $j\in\{0,\dots,\kappa\}$ be fixed, and we let~$x(j)$ be a fixed outcome of~$\cX(j)$, thus determining~$V_{j}$ and~$G_{j}$.

\noindent (i): Fix~$g\in\cF$.
It is clear that $g(V)-\kappa\leq g(V_{j})\leq g(V)$.
Relaxing the upper bound and recalling that $g(V)\geq n/2$, we obtain $g(V_{j})=(1\pm 2\kappa/n)g(V)$.
Note that\COMMENT{$2n^{-1/2}<n^{-1/4}(1-\kappa/n)-\kappa/n \Leftrightarrow 2< n^{1/4}(1-\kappa/n)-1$.}
$2\kappa/n=2n^{-1/2}<n^{-1/4}(1-\kappa/n)-\kappa/n\leq n^{-1/4}(1-j/n)-j/n$, which implies~(i).\COMMENT{$g(V_{j})<(1+2\kappa/n)g(V)$ implies $g(V_{j})<\left(1+n^{-1/4}(1-j/n)-j/n\right)g(V)=\left(1+n^{-1/4}\right)\frac{n-j}{n}g(V)$.
The lower bound proceeds similarly, as $g(V_{j})>(1-2\kappa/n)g(V)$ implies
\begin{eqnarray*}
g(V_{j}) & > & (1-n^{-1/4}(1-j/n)+j/n)g(V) \\ & > & (1-n^{-1/4}(1-j/n)-j/n)g(V) =(1-n^{-1/4})\frac{n-j}{n}g(V).
\end{eqnarray*}.}

\noindent (ii): Let~$\unord{M}\in\unordsubs{V_{j}}{k-1}$.
By~(i), we have
\begin{eqnarray}\label{eq:dgj}
d_{G_{j}}(\unord{M}) & = & g_{\unord{M}}(V_{j}) \geq \left(1-n^{-1/4}\right)\frac{n-j}{n}g_{\unord{M}}(V) \geq\left(1-n^{-1/4}\right)\left(\frac{1}{2}+\gamma\right)(n-j) \nonumber \\ & \geq & \left(\frac{1}{2}+\frac{\gamma}{2}\right)|V_{j}|.
\end{eqnarray}
Since~(\ref{eq:dgj}) holds for all~$\unord{M}\in\unordsubs{V_{j}}{k-1}$, we conclude that~$G_{j}$ is~$\gamma/2$-Dirac.

\noindent The calculations for (iii) and (iv) are straightforward.\COMMENT{\noindent (iii): Fix $e\in E(G_{j})$.
Since~$\weighting{x}$ is $C$-normal, we have
\[
\weighting{x}|_{G_{j}}(e)\geq \frac{1}{Cn^{k-1}} \geq \frac{1}{2C(n-j)^{k-1}}.
\]
Upper bound: $\weighting{x}|_{G_{j}}(e)\leq\frac{C}{n^{k-1}}\leq\frac{2C}{(n-j)^{k-1}}$. We turn to (iv): Fix $v\in V_{j}$.
Note that there are at most~$\kappa\binom{n-2}{k-2}$ edges in $E(G)\setminus E(G_{j})$ containing~$v$ and some vertex $u\in V\setminus V_{j}$.
Thus,
\[
\sum_{\substack{e\ni v \\ e\in E(G_{j})}}\weighting{x}|_{G_{j}}(e) \stackrel{(\ref{eq:doublebound})}{\geq} 1-\kappa\binom{n-2}{k-2}\frac{C}{n^{k-1}}\geq 1-\frac{C\kappa}{(k-2)!n}\geq1-n^{-2/5}.
\]
This completes the proof.}
\end{proof}
For any $j\in\{0,\dots,\kappa\}$ and any fixed outcome $x(j)=(x_{-(k-2)},\dots, x_{j})$ of~$\cX(j)$, we write~$\pr_{x(j)}$ for the probability measure in the conditional probability space where we have fixed $\cX(j)=x(j)$, so that $\probxj{\,\cdot\,}=\prob{\,\cdot\mid\cX(j)=x(j)}$.
We are now ready to prove Lemma~\ref{goodwalk}.
\lateproof{Lemma~\ref{goodwalk}}
We need to prove that, with probability at least $1-1/n$,
\begin{equation}\label{eq:errorg}
\text{Error}_{g}(\cX)\coloneqq\left|\left(\sum_{j=-(k-2)}^{\kappa}g(X_{j})\right)-\frac{\kappa}{n}g(V)\right|<n^{3/10}
\end{equation}
holds simultaneously for every~$g\in\cF$.
It will suffice to prove that~(\ref{eq:errorg}) holds for any fixed $g\in\cF$ with probability at least $1-1/n^{k}$, say,
since $|\cF|\leq n^{k-1}$.
Fix $g\in\cF$ and set $q\coloneqq (\ln n)^{2}$.
By breaking up~$\text{Error}_{g}(\cX)$ and repeatedly applying the triangle inequality, we obtain\COMMENT{$g(V)/n\leq 1$ (and recall that~$g$ is an indicator function) so $|g(X_{j})-g(V)/n|\leq 1.$}:
\begin{eqnarray*}
\text{Error}_{g}(\cX) & = & \left|\sum_{j=-(k-2)}^{0}g(X_{j})+\sum_{j=1}^{q-1}\left(g(X_{j})-\frac{g(V)}{n}\right)+\sum_{j=q}^{\kappa}\left(g(X_{j})-\frac{g(V)}{n}\right)\right| \\ & \leq & 2q+\left|\sum_{j=0}^{\kappa-q}\left(g(X_{j+q})-\frac{g(V)}{n}\right)\right| \\ & \leq & 2q +\left|\sum_{j=0}^{\kappa-q}(g(X_{j+q})-\expn{g(X_{j+q})\mid\cX(j)})\right| \\ & + & \sum_{j=0}^{\kappa-q}\left|\expn{g(X_{j+q})\mid\cX(j)}-\frac{g(V_{j})}{n-j}\right| +\sum_{j=0}^{\kappa-q}\left|\frac{g(V_{j})}{n-j}-\frac{g(V)}{n}\right|.
\end{eqnarray*}
We now prove an upper bound for each of the three sums in the final expression above.
To this end, fix $j\in\{0,\dots, \kappa-q\}$, fix an outcome~$x(j)=(x_{-(k-2)},\dots, x_{j})$ of~$\cX(j)$, and let $\unord{T}\in\unordsubs{V}{k-1}$ be such that $g=g_{\unord{T}}$.
We apply Proposition~\ref{deterministic} to deduce that~$G_{j}$ is $\gamma/2$-Dirac, and that $\weighting{x}|_{G_{j}}$ is $2C$-normal and $n^{-2/5}$-almost-perfect.
We can now apply Lemma~\ref{uniformity} to~$G_{j}$\COMMENT{(with $n-j, 2C, (x_{j-(k-2)},\dots, x_{j})$ playing the roles of~$n, C, \ordered{S}$ respectively, using $(\ln(n-j))^{2}\leq(\ln n)^{2}\leq (n-j)^{1/5}$ and $n^{-2/5}\leq (n-j)^{-2/5}$, and noting $|N_{G_{j}}(\unord{T})|>n^{3/4}$ by Proposition~\ref{deterministic}(i))} to deduce that
\begin{eqnarray*}
\left|\expn{g(X_{j+q})\mid\cX(j)=x(j)}-\frac{g(V_{j})}{n-j}\right| & = & \left|\probxj{X_{j+q}\in N_{G_{j}}(\unord{T})}-\frac{|N_{G_{j}}(\unord{T})|}{n-j}\right| \\ & < & (n-j)^{-13/10}|N_{G_{j}}(\unord{T})| < n^{-5/4}g(V).
\end{eqnarray*}
We deduce\COMMENT{Note that clearly $|\expn{g(X_{j+q})\mid\cX(j)}-g(V_{j})/(n-j)|$ is a random variable inheriting randomness from~$\cX(j)$, and outputting a real number.
If the output of this random variable is at most $n^{-5/4}g(V)$ for all outcomes~$x(j)$ of~$\cX(j)$, then the random variable is deterministically bounded above by $n^{-5/4}g(V)$.} that $|\expn{g(X_{j+q})\mid\cX(j)}-g(V_{j})/(n-j)|<n^{-5/4}g(V)$ for each $j\in\{0,\dots, \kappa-q\}$.
Next, we apply Proposition~\ref{deterministic}(i) to obtain that, for each $j\in\{0,\dots, \kappa-q\}$, we have $|g(V_{j})/(n-j)-g(V)/n|\leq n^{-5/4}g(V)$.
Finally, applying Lemma~\ref{azuma} with~$\log n$ playing the role of~$t$, and using $||g||_{\infty}=1$ (there are no isolated $(k-1)$-tuples), we see that with probability at least $1-2q\exp(-(\log n)^{2}/2)$, we have
\[
\left|\sum_{j=0}^{\kappa-q}(g(X_{j+q})-\expn{g(X_{j+q})\mid\cX(j)})\right|\leq\log n\sqrt{q(\kappa-q)},
\]
so that altogether, with probability at least $1-1/n^{k}$, we have\COMMENT{Clearly $2(\kappa-q+1)n^{-5/4}g(V) < 2\kappa n^{-5/4}g(V)=2n^{-3/4}g(V)\leq 2n^{1/4}<\frac{1}{3}n^{3/10}$.
Next, clearly $2q =2(\ln n)^{2}<\frac{1}{3}n^{3/10}$.
Finally, $\log n\sqrt{q(\kappa-q)}\leq\log n\sqrt{q\kappa}=\log n \ln n n^{1/4}<\frac{1}{3}n^{3/10}$.}
\[
\text{Error}_{g}(\cX) \leq 2q + \log n\sqrt{q(\kappa-q)} +2(\kappa-q+1)n^{-5/4}g(V) < n^{3/10},
\]
completing the proof of the lemma.
\endproof
We now have all the tools we need to prove Lemma~\ref{iteration}.
\lateproof{Lemma~\ref{iteration}}
Choose a new constant~$C$ satisfying $1/n\ll c\ll1/C\ll\gamma,1/k$, and let~$\weighting{x}\colon E(G)\rightarrow\bR^{+}$ be a $C$-normal perfect fractional matching (such an~$\weighting{x}$ exists by Lemma~\ref{normality}).
Write $\kappa\coloneqq\sqrt{n}$, and let $\cX=(X_{-(k-2)},\dots,X_{\kappa})$ be the self-avoiding $\weighting{x}$-walk on~$G$ with starting tuple~$\ordered{S}$.
It is clear from the definition of a self-avoiding $\weighting{x}$-walk that any outcome of~$\cX$ corresponds to a $\kappa$-path in~$G$ with~$\ordered{S}$ as one end (note that the walk does not stop before time~$\kappa$, since all codegrees are large enough).
We argue now that good outcomes of~$\cX$ also satisfy condition (iii),
where we say an outcome~$X=(X_{-(k-2)}, \dots, X_{\kappa})$ of~$\cX$ is a \textit{good outcome} if~$X$ satisfies~(\ref{eq:gooddef}).
Let~$P$ be a tight path in~$G$ corresponding to a good outcome~$X$ of~$\cX$, let~$\ordered{T}$ be the non-$\ordered{S}$ end of~$P$, and let~$G_{\kappa}$ denote the residual graph of~$G$ at time~$\kappa$ of~$X$.
Thus $G_{\kappa}=G-(V(P)\setminus\unord{T})$ and $|V_{\kappa}|=n-\kappa$.
Let $\unord{M}\in\unordsubs{V_{\kappa}}{k-1}$ and let $g_{\unord{M}}\in\cF$ be the tracking function of~$G$ corresponding to~$\unord{M}$.
Since~$X$ is good, we obtain:\COMMENT{Note $\frac{n^{3/10}}{|V_{\kappa}|}=\frac{n^{3/10}}{n-\sqrt{n}}\leq 2n^{-7/10}\leq n^{-2/3}$.}
\begin{eqnarray}\label{eq:penult}
d_{G_{\kappa}}(\unord{M}) & = & g_{\unord{M}}(V_{\kappa})= g_{\unord{M}}(V)-\sum_{j=-(k-2)}^{\kappa}g_{\unord{M}}(X_{j}) +\sum_{j=\kappa-(k-2)}^{\kappa}g_{\unord{M}}(X_{j}) \nonumber \\ & \stackrel{(\ref{eq:gooddef})}{\geq} & \frac{n-\kappa}{n}g_{\unord{M}}(V)-n^{3/10}\geq\left(\frac{1}{2}+\gamma-n^{-2/3}\right)|V_{\kappa}|.
\end{eqnarray}
Since~(\ref{eq:penult}) holds for all~$\unord{M}\in\unordsubs{V_{\kappa}}{k-1}$, we conclude that $\delta(G_{\kappa})\geq(1/2+\gamma-n^{-2/3})|V_{\kappa}|$.

Lastly then, it suffices to count the number of good outcomes of~$\cX$.
We begin by finding an upper bound for the probability that~$\cX$ yields any particular fixed tight path.
For any $j\in\{0,\dots,\kappa\}$, we have by Proposition~\ref{deterministic}(ii)--(iii) that~$G_{j}$ is $\gamma/2$-Dirac and~$\weighting{x}|_{G_{j}}$ is $2C$-normal.
It follows\COMMENT{$\frac{\max_{e\in E(G_{j})}\weighting{x}|_{G_{j}}(e)}{\min_{\unord{S}\in\unordsubs{V_{j}}{k-1}}\sum_{v\in V_{j}}\weighting{x}|_{G_{j}}(\unord{S}\cup\{v\})}\leq \frac{2C}{(n-j)^{k-1}}\cdot\frac{2C(n-j)^{k-1}}{(1/2+\gamma/2)(n-j)}\leq \frac{8C^{2}}{n-j}$.} that~$\weighting{x}|_{G_{j}}$ is $8C^{2}/(n-j)$-upper-balanced.
In particular, setting $p\coloneqq 16C^{2}/n$, we have that all transition probabilities of~$\cX$ are bounded from above by~$p$.
Let $Q=(q_{-(k-2)}, \dots, q_{\kappa})$ be a fixed $\kappa$-path in~$G$ with $\ordered{S}=(q_{-(k-2)},\dots, q_{0})$ as one end.
Then
\[
\prob{\cX=Q}=\prod_{j=1}^{\kappa}\prob{X_{j}=q_{j}\bigg|\bigcap_{i=-(k-2)}^{j-1}\{X_{i}=q_{i}\}}\leq p^{\kappa}.
\]
By Lemma~\ref{goodwalk}, we have that $\prob{\cX \hspace{1mm}\text{is good}}\geq1/2$, so we conclude that the number of good outcomes of~$\cX$ (and thus the number of tight $\kappa$-paths in~$G$ satisfying (ii) and~(iii)) is at least $(1/2)/p^{\kappa}\geq (cn)^{\kappa}$, which completes the proof of the lemma.
\endproof
\section{Counting and absorbing long paths}\label{countinglong}
In this section we show how to iterate Lemma~\ref{iteration} to construct tight paths in~$G$ which use almost all of the vertices of~$G$, and we count the number of choices that can be made in this process to obtain a lower bound for the number of these long paths.
Finally, we prove Theorem~\ref{hamcycs} by showing how these paths can be completed into tight Hamilton cycles of~$G$.
\begin{lemma}\label{algorithmic}
Let $\gamma>0$, let~$k\geq 2$, and let~$G$ be an $(n,k,\gamma)$-graph.
There are at least $\exp(n\ln n-\Theta(n))$ tight paths in~$G$ of length at least $n-n^{7/8}$.
\end{lemma}
\begin{proof}
We describe an algorithm on~$G$.
Let $\ordscript{S}{0}\in\ordsubs{V}{k-1}$ be arbitrary, set $G_{0}\coloneqq G$, set $n_{0}\coloneqq n$, and set $\gamma_{0}\coloneqq \gamma$. For each $i\geq 0$, set $n_{i+1}\coloneqq n_{i}-\sqrt{n_{i}}$, and set $\gamma_{i+1}\coloneqq \gamma_{i}-(n_{i})^{-2/3}$.
Set~$L$ to be the smallest index such that $n_{L}<n^{7/8}$.
Suppose we have already performed~$i$ steps of the algorithm, and obtained a $k$-graph~$G_{i}$ on~$n_{i}$ vertices satisfying $\delta(G_{i})\geq (1/2 +\gamma_{i})n_{i}$, and we have obtained $\ordscript{S}{i}\in\ordsubs{V(G_{i})}{k-1}$.
If~$\gamma_{i}<\gamma/2$ or $i=L$, then we terminate the algorithm.
Otherwise, we apply Lemma~\ref{iteration} to~$G_{i}$ to obtain a set~$\cP_{i+1}$ of $\sqrt{n_{i}}$-paths, each with chosen starting tuple~$\ordscript{S}{i}$.
Choose $P_{i+1}\in\cP_{i+1}$ arbitrarily, let~$\ordscript{T}{i+1}$ be the non-$\ordscript{S}{i}$ end of~$P_{i+1}$, and put~$\ordscript{S}{i+1}\coloneqq\revscript{T}{i+1}$.
Set $G_{i+1}\coloneqq G_{i}-(V(P_{i+1})\setminus \unord{S}_{i+1})$.
Observe that by Lemma~\ref{iteration}(iii), we have $\delta(G_{i+1})\geq (1/2+\gamma_{i+1})n_{i+1}$.

Let $r_{i}\coloneqq \sum_{j=0}^{i-1}(n_{j})^{-2/3}$.
Note that, provided $\gamma_{i-1}\geq\gamma/2$, we have $\gamma_{i}=\gamma-r_{i}$.
We claim that the algorithm does not terminate in the first~$L$ steps.
To see this, note that it suffices to show that $r_{L}=o(1)$.
Write $\kappa_{i}\coloneqq \sqrt{n_{i}}$ and observe that $n_{L-1}=n-\sum_{j=0}^{L-2}\kappa_{j}\leq n-(L-1)\kappa_{L-1}$.
Re-arranging, we obtain that $L\leq 2n/\kappa_{L-1}$.
Using $n_{L-1}\geq n^{7/8}$, we obtain that
\[
r_{L}\leq \frac{L}{(n_{L-1})^{2/3}}\leq\frac{2n}{(n_{L-1})^{7/6}}\leq 2n^{-1/48}=o(1),
\]
so the algorithm does not terminate in the first~$L$ steps, as claimed.
When the algorithm terminates, we have obtained tight paths $P_{1},\dots, P_{L}$.
By construction, we may concatenate these paths, in order, to obtain a path $Q\coloneqq \bigcup_{i\leq L}P_{i}$ of length $n-n_{L}\geq n-n^{7/8}$.
Let~$N$ be the number of tight paths of length~$n-n_{L}$ in~$G$.
By Lemma~\ref{iteration}(i), there is a positive constant~$c<1$ such that the number of choices for~$P_{i+1}$ is at least~$(cn_{i})^{\kappa_{i}}$, for each $i\in\{0,\dots,L-1\}$.
Thus, we obtain
\[
N \geq \prod_{i=0}^{L-1}(cn_{i})^{\kappa_{i}} \geq c^{n}\prod_{i=0}^{L-1}\frac{n_{i}!}{n_{i+1}!}=c^{n}\frac{n!}{n_{L}!}\geq c^{n}\frac{n!}{(n^{7/8})!} = \exp(n\ln n-\Theta(n)).
\]
\end{proof}
We are now ready to prove Theorem~\ref{hamcycs}.
\lateproof{Theorem~\ref{hamcycs}}
The upper bound holds trivially.
To prove the lower bound, we choose a set $W\subseteq V$ of size~$n^{9/10}$ uniformly at random.
A simple application\COMMENT{
For each $\unord{S}\in\unordsubs{V}{k-1}$, put $X_{\unord{S}}=|N_{G}(\unord{S})\cap W|$.
Then $X_{\unord{S}}\sim\text{hyp}(|W|, d_{G}(\unord{S}),n)$, so that $\expn{X_{\unord{S}}}\geq (1/2 +\gamma)|W|$.
Applying Lemma~\ref{chernoff} we see $\prob{X_{\unord{S}}\leq(1/2+\gamma/2)|W|}\leq\prob{X_{\unord{S}}\leq\expn{X_{\unord{S}}}-(\gamma/2)|W|} \leq \exp(-\gamma^{2}|W|/8)=\exp(-\Theta(n^{9/10}))$.
Lastly note $n^{k-1}\exp(-\Theta(n^{9/10}))=\exp((k-1)\ln n -\Theta(n^{9/10}))=o(1)$.}
of Lemma~\ref{chernoff} shows that there is a choice of~$W$ such that $|N(\unord{S})\cap W|\geq(1/2 + 3\gamma/4)|W|$ for all $\unord{S}\in\unordsubs{V}{k-1}$.
Fix such a choice of~$W$, set $G'\coloneqq G-W$, and put $n'\coloneqq n-n^{9/10}$.
Then~$G'$ is $\gamma/2$-Dirac, and we apply Lemma~\ref{algorithmic} to~$G'$ (with~$\gamma/2$ playing the role of~$\gamma$) to find a set~$\cP$ of tight paths of length at least $n' - (n')^{7/8}$ in~$G'$, such that $|\cP|\geq\exp(n'\ln n' -\Theta(n'))=\exp(n\ln n - \Theta(n))$.
Fix $P\in\cP$, let~$\ordscript{S}{P}$ and~$\ordscript{T}{P}$ be the ends of~$P$, and let $U_{P}\coloneqq V(G')\setminus V(P)$, so that $|U_{P}|\leq (n')^{7/8}\leq n^{7/8}$.
Notice that~$G[W\cup U_{P}\cup\unord{S}_{P}\cup\unord{T}_{P}]$ is $\gamma/2$-Dirac.
Thus, by Lemma~\ref{hampath}, there is a tight Hamilton path~$Q_{P}$ of~$G[W\cup U_{P}\cup\unord{S}_{P}\cup\unord{T}_{P}]$ with ends~$\revscript{S}{P}$ and~$\revscript{T}{P}$.
Then $C_{P}\coloneqq P\cup Q_{P}$ is a tight Hamilton cycle of~$G$.

Define $\cC\coloneqq\{C_{P}\colon P\in\cP\}$ and note that for each $C\in\cC$, the number of $P\in\cP$ with $C=C_{P}$ is at most~$n^{2}$, since~$P$ must be a subpath of~$C$.
We conclude that~$\cC$ is a set of at least $n^{-2}\exp(n\ln n-\Theta(n))=\exp(n\ln n-\Theta(n))$ tight Hamilton cycles in~$G$.
\endproof
Finally, we prove Corollary~\ref{hamellcycles}.
\lateproof{Corollary~\ref{hamellcycles}}
Let $\gamma>0$ be fixed, let $k\geq2$, $\ell\in\{0,\dots,k-1\}$, $(k-\ell)\mid n$, and let~$G$ be a $k$-graph on~$n$ vertices satisfying $\delta(G)\geq (1/2+\gamma)n$.
Firstly, suppose $\ell=0$, and recall that the number of Hamilton $0$-cycles of~$G$ is precisely the number of perfect matchings of~$G$.
By considering the number of perfect matchings in the complete $k$-graph on~$n$ vertices, it is easy to see that the upper bound of (i) holds.\COMMENT{Suppose $k\mid n$, and let~$m$ be the number of perfect matchings of the complete $k$-graph on~$n$ vertices.
Then observe that
\[
m=\frac{\binom{n}{k}\binom{n-k}{k}\dots\binom{k}{k}}{\left(\frac{n}{k}\right)!}=\frac{n!}{(k!)^{n/k}\left(\frac{n}{k}\right)!}=\exp\left(\left(1-\frac{1}{k}\right)n\ln n -\Theta(n)\right).
\]}
We now use Theorem~\ref{hamcycs} to show that the lower bound holds.
Let~$\cM$ be the set of perfect matchings of~$G$, and let~$\cC$ be the set of tight Hamilton cycles of~$G$.
Notice that, for any $M\in\cM$, there are at most~$(n/k)!(k!)^{n/k}$ choices of $C\in\cC$ such that $M\subseteq E(C)$, because we may construct all vertex orderings corresponding to possible such~$C$ by reordering the edges of~$M$ and the vertices within them.
By applying Theorem~\ref{hamcycs}, we conclude that
\[
|\cM|\geq\frac{|\cC|}{\left(\frac{n}{k}\right)!(k!)^{n/k}} = \exp\left(\left(1-\frac{1}{k}\right)n\ln n -\Theta(n)\right).
\]
For the case $\ell\in[k-1]$, firstly notice that~$\frac{k-\ell}{2}(n-1)!=\exp(n\ln n-\Theta(n))$ is a trivial upper bound for the number of Hamilton $\ell$-cycles of~$G$. Finally, it suffices to apply Theorem~\ref{hamcycs} to~$G$ and observe that every tight Hamilton cycle of~$G$ contains~$k-\ell$ Hamilton $\ell$-cycles (since $(k-\ell)\mid n$), and each Hamilton $\ell$-cycle of~$G$ is contained in at most $(k!)^{n/(k-\ell)}$ tight Hamilton cycles.
\endproof
\section{Concluding remarks}
Though Theorem~\ref{hamcycs} holds in $\gamma$-Dirac $k$-graphs with equality, we believe that the error bound can be made more precise.
More specifically, we believe the following hypergraph version of~\cite[Theorem 1.1]{CK09b} holds, giving a more accurate lower bound for the number of tight Hamilton cycles in such hypergraphs.
\begin{conj}\label{hamcycconj}
For a fixed integer $k\geq 2$ and a fixed constant $\gamma>0$, the number of tight Hamilton cycles of a $k$-graph~$G$ on~$n$ vertices with $\delta(G)\geq(1/2+\gamma)n$ is at least $(1/2-o(1))^{n}n!$.
\end{conj}
It would of course be desirable to obtain a formula for the number of tight Hamilton cycles in $\gamma$-Dirac $k$-graphs~$G$ which takes properties of~$G$ like the degrees and codegrees into account.
We recall that such a formula has already been obtained~\cite[Theorem 1.3, Theorem 1.5]{CK09b} in terms of the `entropy of~$G$' in the $k=2$ case, and it would be interesting to see if this (or a similar) notion can be extended to $k\geq 3$.

Finally, we note that the results of~\cite{CK09b} show that graphs with minimum degree precisely at the threshold for Hamiltonicity in fact have many Hamilton cycles.
The exact minimum codegree threshold for existence of a tight Hamilton cycle in $k$-graphs on~$n$ vertices is not yet known for $k\geq 4$, but is known to be~$\lfloor n/2\rfloor$ in the case $k=3$~\cite[Theorem 1.2]{RRS11}, and it is of course a natural question to ask if the conclusions of Theorem~\ref{hamcycs} or indeed Conjecture~\ref{hamcycconj} hold in this exact setting.

\bibliographystyle{amsplainv2}

\bibliography{References}

\APPENDIX{
\renewcommand{\thetheorem}{\Alph{theorem}}
\setcounter{theorem}{0}
\appendix\section*{Appendix: Proof of Proposition~\ref{tidytool}}
We first need the following Chernoff-type bound (see~\cite{CL06} and~\cite{JLR00} for example), and a well-known result on the probability that a binomially distributed random variable assumes its mean value.
\begin{lemma}\label{weightedchernoff}
Let $X_{1},\dots, X_{n}$ be independent Bernoulli random variables with $\prob{X_{i}=1}=p_{i}$ for each~$i\in [n]$.
Let $a_{1},\dots,a_{n}\geq 0$ with $\sum_{i=1}^{n}a_{i}>0$, set $X=\sum_{i=1}^{n}a_{i}X_{i}$, and define $\nu\coloneqq\sum_{i=1}^{n}a_{i}^{2}p_{i}$.
Then $\prob{X\leq\expn{X}-t}\leq e^{-t^{2}/(2\nu)}$.
\end{lemma}
\begin{lemma}\label{littlehelper}
Let $1/n \ll 1/m \leq 1$, with $m,n\in\bN$, and let~$X$ be a binomial random variable with parameters~$n$ and~$p\coloneqq m/n$.
Then $\pr[X=m]\geq 1/(4\sqrt{m})$.\COMMENT{I prove this with asymptotic notation, which suffices as we can just take~$n$ large enough in comparison to~$m$. So suppose we have $m^{2}=o(n)$. Note firstly that $m!\leq e^{1-m}m^{m+1/2}$ and $(1-p)^{n}=e^{-np+O(np^{2})}=e^{-m+O(m^{2}/n)}=e^{-m+o(1)}$. Now observe that $\prob{X=m} =\binom{n}{m}p^{m}(1-p)^{n-m}\geq m!^{-1}(n-m)^{m}p^{m}(1-p)^{n-m}=m!^{-1}(p(n-m)/(1-p))^{m}(1-p)^{n}=m!^{-1}m^{m}(1-p)^{n}\geq e^{m-1}m^{-m-1/2}m^{m}e^{-m+o(1)}=e^{-1+o(1)}/\sqrt{m}\geq1/4\sqrt{m}$.}
\end{lemma}
We now prove Proposition~\ref{tidytool}.
\lateproof{Proposition~\ref{tidytool}}
For any subset $\cQ\subseteq\cP$, let~$V(\cQ)$ denote the set of all vertices in any $\ell$-set $A\in\cQ$.
Define $p\coloneqq m/|\cP\setminus\cP_{0}|$.
We construct a random set $X\subseteq\cP\setminus\cP_{0}$ by including each $\ell$-set $A\in\cP\setminus\cP_{0}$ independently with probability~$p$.
Let~$Y\coloneqq X\cup\cP_{0}$, and define the events $\cE_{1}\coloneqq\{|X|=m\}$ and $\cE_{2}\coloneqq\bigcap_{\unord{S}\in\unordsubs{V(Y)}{k-1}}\{d_{G[V(Y)]}(\unord{S})\geq(1/2+\gamma/2)\ell(m+t)\}$.
Let~$\mathrm{\mathbb{P}}_{b}$ be the probability measure for the space corresponding to constructing~$X$.
Notice then that $\prob{\,\cdot\,}=\probb{\,\cdot\mid\cE_{1}}$.
It remains to prove that $\probb{(\cE_{2})^{c}\mid\cE_{1}}\leq e^{-\sqrt{m}}$.
Write $M\coloneqq\binom{V(G)}{k-1}$, and for each $\unord{S}\in M$ write $d_{Y}(\unord{S})\coloneqq |N_{G}(\unord{S})\cap V(Y)|$, write $\cP(\unord{S})\coloneqq\{A\in\cP\colon \unord{S}\cap A\neq\emptyset\}$, write $\cP_{\unord{S}}\coloneqq\cP\setminus(\cP_{0}\cup\cP(\unord{S}))$, and write $J_{\unord{S}}\coloneqq N_{G}(\unord{S})\cap V(\cP_{\unord{S}})$.
Notice that $d_{Y}(\unord{S})\geq d'_{Y}(\unord{S})\coloneqq|N_{G}(\unord{S})\cap V(Y)\cap V(\cP_{\unord{S}})|$.
Fix $\unord{S}\in M$.
Notice that $|J_{\unord{S}}|\geq (1/2+3\gamma/4)n$, since $|V(\cP_{0}\cup\cP(\unord{S}))|\leq \ell(t+k)$.
Observe that\COMMENT{Note that $np=nm/(n/\ell-t) =m\ell/(1-t\ell/n)\geq m\ell$. (Similarly $np\leq 2m\ell$.)}
\[
\expnb{d'_{Y}(\unord{S})}=|J_{\unord{S}}|p\geq(1/2+3\gamma/4)np\geq(1/2+2\gamma/3)\ell(m+t).
\]
For each $\ell$-set $A\in\cP_{\unord{S}}$, let~$Y_{A}$ be the indicator random variable for the event~$\{A\in X\}$, and let $c_{A}\coloneqq|A\cap N_{G}(\unord{S})|$.
Then we have $d'_{Y}(\unord{S})=\sum_{A\in\cP_{\unord{S}}}c_{A}Y_{A}$, and by applying Lemma~\ref{weightedchernoff} to~$d'_{Y}(\unord{S})$, we obtain\COMMENT{Notice that $\sum_{A}c_{A}^{2}p\leq (n/\ell)\ell^{2}p\leq 2m\ell^{2}$. Applying Lemma~\ref{weightedchernoff} to~$d'_{Y}(\unord{S})$, we obtain $\probb{d'_{Y}(\unord{S})<(1/2+\gamma/2)\ell(m+t)}\leq\probb{d'_{Y}(\unord{S})\leq\expnb{d'_{Y}(\unord{S})}-\gamma \ell(m+t)/6} \leq \exp(-\gamma^{2} (m+t)^{2}\ell^{2}/144m\ell^{2}) \leq e^{-3\sqrt{m}}$.}
$\probb{d'_{Y}(\unord{S})<(1/2+\gamma/2)\ell(m+t)}\leq e^{-3\sqrt{m}}$.
Note that for any $\unord{S}\in M$ and any $v\in J_{\unord{S}}$, the events~$\{\unord{S}\subseteq V(Y)\}$ and~$\{v\in V(Y)\}$ are independent by construction.
Let $\Tilde{\cP}(\unord{S})\coloneqq\cP(\unord{S})\setminus\cP_{0}$, and for each $0\leq j\leq k-1$, let $M_{j}\coloneqq\{\unord{S}\in M\colon |\Tilde{\cP}(\unord{S})|=j\}$.
Note that for each $\unord{S}\in M_{j}$, we have $\prob{\unord{S}\subseteq V(Y)}=p^{j}$, and note further that $|M_{j}|\leq (\ell t)^{k}n^{j}$, for each~$j$.
Then by a union bound over all $\unord{S}\in M$ we obtain\COMMENT{Note that $\probb{\unord{S}\subseteq V(Y), d_{Y}(\unord{S})<(1/2+\gamma/2)\ell(m+t)}\leq\probb{\unord{S}\subseteq V(Y), d'_{Y}(\unord{S})<(1/2+\gamma/2)\ell(m+t)}$ for each~$\unord{S}$, since the former event implies the latter.}
\begin{eqnarray*}
\probb{(\cE_{2})^{c}} & \leq & \sum_{j=0}^{k-1}\sum_{\unord{S}\in M_{j}}\probb{\unord{S}\subseteq V(Y), d'_{Y}(\unord{S})<(1/2+\gamma/2)\ell(m+t)}\\ & \leq & e^{-3\sqrt{m}}\sum_{j=0}^{k-1}(\ell t)^{k}(np)^{j} \leq e^{-3\sqrt{m}}\sum_{j=0}^{k-1}(\ell t)^{k}(2m\ell)^{j} \leq e^{-2\sqrt{m}}.
\end{eqnarray*}
Finally, by Lemma~\ref{littlehelper} we have $\probb{\cE_{1}}\geq 1/(4\sqrt{m})$, so we conclude that $\probb{(\cE_{2})^{c}\mid\cE_{1}}\leq\probb{(\cE_{2})^{c}}/\probb{\cE_{1}}\leq e^{-\sqrt{m}}$.
\endproof}

\small
\vskip12mm

\noindent
Stefan Glock
{\tt <stefan.glock@eth-its.ethz.ch>}\\
ETH Institute for Theoretical Studies\\
Clausiusstrasse 47\\
8092 Z\"urich\\
Switzerland\\

\small
\noindent
Stephen Gould, Daniela K\"uhn, Deryk Osthus
{\tt <\{spg377, d.kuhn, d.osthus\}@bham.ac.uk>}\\
School of Mathematics\\
University of Birmingham\\
Edgbaston\\
Birmingham\\
B15 2TT\\
United Kingdom\\

\small
\noindent
Felix Joos
{\tt <joos@informatik.uni-heidelberg.de>}\\
Institut f\"{u}r Informatik\\
Heidelberg University\\
Germany\\
\end{document}